\documentclass[a4paper, 12pt]{amsart}

\usepackage[
    a4paper
    ,hmargin={2.54cm,2.54cm}
    ,vmargin={3.17cm,3.17cm}
]{geometry}

\usepackage{amsmath, amsthm, amssymb, amsfonts, appendix, bm, comment, commutative-diagrams, enumerate, geometry, graphicx, mathrsfs, indentfirst, pgfplots, url, verbatim}
\usepackage{extarrows}

\usepackage[runin]{abstract}

\usepackage[
    backend=biber
    ,url=false
    ,doi=false
    ,eprint=true 
    ,isbn=false 
    ,maxnames=99
    ,maxbibnames=99
]{biblatex}
\AtEveryBibitem{\clearfield{month}}
\AtEveryBibitem{\clearfield{day}}

\usepackage[
    pdfborder={0 0 0}
    ,colorlinks=true
    ,linkcolor=blue
    ,CJKbookmarks=true
]{hyperref}
\pgfplotsset{compat=1.17}

\newtheorem{theorem}{Theorem}
\newtheorem{lemma}{Lemma}
\newtheorem{proposition}{Proposition}
\newtheorem{definition}{Definition}

\newtheorem{remark}{Remark}

\usepackage{tikz}
\usepackage{lmodern} 
\tikzset{every picture/.style={line width=0.75pt}} 

\numberwithin{equation}{section}

\title[Inverse Quasilinear Conductivities of CTA Manifolds]{The Calderón Problem for Quasilinear Conductivities of Conformally Transversally Anisotropic Media} 

\author[X. Chen]{Xi Chen}
\address{{Shanghai Center for Mathematical Sciences, Fudan University, Shanghai 200438, China;	
		Center for Applied Mathematics, Fudan University, Shanghai 200433, China;	
		School of Mathematical Sciences, Fudan University, Shanghai 200433, China. } }
\email{xi\_chen@fudan.edu.cn}

\author[Z. Jin]{Ziyun Jin}
\address{Shanghai Center for Mathematical Sciences, Fudan University, Shanghai 200438, China. }
\email{22110180023@m.fudan.edu.cn}

\date{}

\begin{document}

\maketitle

\begin{abstract}
   This paper investigates Calderón's problem on a conformally transversally anisotropic manifold $ (M,g) $ of dimension $n \geq 3$, where the conductivity $ a(s,x,p) $ might depend on both the electric potential and the electric field. We establish that for all $(t,x)\in \mathbb{R}\times M$ and $\beta \in \mathbb{N}^{1+n}$ the derivatives $ \partial_{(s,p)}^\beta a(s,x,p)|_{(s,p)=(t,0)}$  are uniquely determined by the boundary voltage-current measurements. If $ a(s,x,p) $ is analytic in $ p $, then $ a(s,x,p) $ can be uniquely recovered. 
\end{abstract}\par

\tableofcontents

\section{Introduction}

The \textit{Electrical Impedance Tomography} (EIT) problem, also known as \textit{Calderón's problem}, is to recover the interior electrical conductivity of a medium by making current and voltage measurements at its boundary. 
\par
Calderón \cite{MR590275} formulated the prototype of the inverse boundary value problem models for the EIT problem. Generally, we let $ \Omega\subset \mathbb{R}^n $ be a bounded domain with smooth boundary $ \partial\Omega $, and consider the following Dirichlet problem of divergence-type quasilinear elliptic equations:
\begin{equation}\label{eq:quasi_gen}
    \left\{\begin{aligned}
    \partial_i(\gamma^{ij} (u,x,\nabla u)\partial_j u)  &= 0 \quad\text{in}~\Omega,\\
    u                     &= f \quad\text{on}~ \partial \Omega,
    \end{aligned}\right.
\end{equation}
where the coefficient 
\begin{equation}\label{def:gamma}
    \gamma:=(\gamma^{ij}(s,x,p))\in C^{2,\alpha}(\mathbb{R}\times \Omega\times\mathbb{R}^n,\text{\rm Mat}_{n}(\mathbb{R}))
\end{equation}
is symmetric and uniformly elliptic, in the sense
\begin{equation}\label{condition:elliptic}
    \gamma^{ij} (\lambda,x,0)\xi_i\xi_j \geq C(\lambda)\delta^{ij} \xi_i\xi_j>0,\quad\forall\lambda\in\mathbb{R},~(\xi_i)_{i=1}^{n}\in\mathbb{R}^n \backslash \{0\}.
\end{equation}
The matrix $ \gamma  $ models the conductivity of the underlying medium, which is anisotropic and possibly depends on the electrical potential and electrical field.\par

The voltage-current measurement is represented by the \textit{Dirichlet-to-Neumann map}. For any $0<\alpha<1$, we denote the ball in $ C^{2,\alpha} (\partial \Omega) $ centered at $ f_0\in C^{2,\alpha} (\partial \Omega) $ with radius $ \delta_0>0 $ by
\begin{equation}\label{eq:def-domain}
    B^{2,\alpha}_{\partial \Omega} (f_0,\delta_0):=\{f \in  C^{2,\alpha} (\partial \Omega): \|f-f_0\|_{C^{2,\alpha} (\partial \Omega)} \leq \delta_0\}.
\end{equation}
A standard argument of implicit function theorem proves that there exists a map $ \delta:\mathbb{R}\to\mathbb{R}^+ $ such that for any $ f\in B^{2,\alpha}_{\partial \Omega} (\lambda,\delta(\lambda)) $, (\ref{eq:quasi_gen}) admits a unique solution $ u^f $ near $ \lambda $. See for example Lemma \ref{lmm:forward}. Then, the corresponding \textit{DN map} (short for Dirichlet-to-Neumann map) is defined as follows:
\begin{displaymath}
    \begin{aligned}
        N_\gamma:  \bigcup_{\lambda\in\mathbb{R}}  B^{2,\alpha}_{\partial \Omega} (\lambda,\delta(\lambda))&\longrightarrow C^{1,\alpha}(\partial \Omega)\otimes\Omega^{n-1}(\partial \Omega),\\ 
        f & \longmapsto \gamma^{ij}\left(u^f(x),x,\nabla u^f(x)\right) \partial_j u^f(x)\nu_i.
    \end{aligned}
\end{displaymath}
Here $ \nu_i=\nu_i(x) $ denotes the differential form $\iota^*(\partial_i \llcorner dx^1\wedge dx^2\wedge\cdots\wedge dx^n) $, where $ \iota:\partial \Omega\to \overline{\Omega} $ is the standard embedding. In local coordinates, $ \nu_i=n_i d\mbox{Area}_{\partial \Omega}  $, where $ n_i $  is the $ i $-th component of the unit outer normal vector on $ \partial \Omega $. Therefore, Calderón's problem reduces to an inverse problem of (\ref{eq:quasi_gen}): is it possible to recover the anisotropic conductivity $ \gamma=(\gamma^{ij}) $ uniquely from the information of $ N_\gamma $? \par
It is well-known that the DN map is coordinate-invariant, and whence the best one can achieve is the unique determination of $ \gamma $ up to the following gauge transform induced by coordinate changing.
\begin{definition}[Gauge transform]
    For a $C^{3,\alpha}$-diffeomorphism $\Phi\in\text{\rm Diff}(\overline{\Omega})$,  the \textbf{gauge transform} $ H_\Phi$ on $C^{2,\alpha}(\mathbb{R}\times \Omega\times\mathbb{R}^n,\text{\rm Mat}_{n}(\mathbb{R}))$ is defined  by
    \begin{equation}\label{def:gauge-transform}
        H_{\Phi}: K(s,x,p)\longmapsto \frac{D\Phi(\cdot)^T K(s,\cdot,D\Phi(\cdot)p)D\Phi(\cdot)}{|\det D\Phi(\cdot)|}\circ\Phi^{-1}(x). 
    \end{equation}
\end{definition}
The DN map is invariant under this gauge transform, that is, $ N_{H_\Phi\gamma}=N_{\gamma} $. See Lemma \ref{lmm:gauge-invariance} for details.\par

If $ \gamma $ is linear and isotropic, i.e., $ \gamma^{ij}(s,x,p)=\gamma(x)\delta^{ij}  $, Calder\'{o}n's problem is well-understood. The uniqueness of inversion is due to \cite{sylvester1987global} for $ n\geq 3 $, and \cite{nachman1996global} for $ n=2 $. See for example \cite{sylvester1988inverse,alessandrini1995local,alessandrini1997examples} for further development.\par 

If $ \gamma $ is linear and anisotropic, i.e., $ \gamma(s,x,p)=\gamma(x) $ has distinct eigenvalues, the EIT problem is far more involved. A key observation made in \cite{lee1989determining} linked anisotropic Calderón's problems with inverse coefficient problems of elliptic equations on a closed Riemannian manifold $ (M,g) $ with boundary $ \partial M $, and of dimension $ n\geq 3 $. Explicitly, the metric $ g $ is given by 
\begin{equation*}
    (g_{ij}) =(\det \gamma)^{\frac{1}{n-2}} (\gamma^{ij})^{-1},~\text{\rm or equivalently}~(\gamma^{ij})=\sqrt{\det g}(g_{ij} )^{-1}.
\end{equation*}
When $ g $ is analytic, the unique retrieval of $ \gamma $ was proved in \cite{lassas2001determining,lassas2003dirichletneumann,lassas2020poisson}. When $ g $ is \textit{conformally transversally anisotropic} (see Definition \ref{def:cta-metric}) in some given conformal class, \cite{dossantosferreira2009limiting,dossantosferreira2016calderon,dossantosferreira2020linearized} reconstructed the conformal factor of $ g $ on certain conditions. In general, the anisotropic Calderón problem remains an open problem.\par

    
The conductivities of a medium might vary as electrical potentials and electric fields change. This corresponds to nonlinear conductivities $ \gamma(s,x,p) $ and nonlinear elliptic equations. For semilinear equations
\begin{equation}\label{eq:semi-linear}
    -\Delta u+a(u,x)=0,
\end{equation}
\cite{isakov1994global} developed the first-order linearization method to identify $ a(s,x) $ with some constraints on $ a $. \cite{lassas2020partial,feizmohammadi2020inverse,lassas2021inverse} recovered $ a(s,x) $ for $ a $ analytic in $ s $, by applying the higher-order linearization scheme originated in \cite{kurylev2018inverse} for nonlinear hyperbolic equations. Furthermore, \cite{munoz2020calderon,carstea2021calderon} extended this framework to inverse problems of quasilinear equation (\ref{eq:quasi_gen}) with $ \gamma^{ij}(s,x,p)=a(s,x,p)\delta^{ij}$, and \cite{liimatainen2024calderon} applied this to the general quasilinear anisotropic case in dimension $ n=2 $. See \cite{salo2022inverse,kian2023partial} for recent progress on partial data problems.\par

For practical applications, it is more meaningful but more challenging to determine anisotropic coefficients in nonlinear equations. For analytic coefficients $ \gamma^{ij}(s,x,p)=\gamma^{ij}(s,x)$, \cite{sun1997inverse} solved inverse problems of (\ref{eq:quasi_gen}) via the second-order linearization. When the electric field is independent of the electric potential, i.e., $ J(s,x,p)=J(x,p) $, \cite{kang2002identification, carstea2019reconstruction} reconstructed the coefficients of the Taylor expansion in $ p $ under various a priori conditions. For more recent results, see \cite{shankar2020recovering,carstea2021inverse,nurminen2023inverse,kholil2024uniqueness,carstea2024inverse}.\par

\par

In this article, we are interested in inverse problems on \textit{conformally transversally anisotropic} (CTA for short) manifolds, introduced by Dos Santos Ferreira, Kenig, Salo, and Uhlmann in \cite{dossantosferreira2009limiting}.
\begin{definition}[CTA manifold]\label{def:cta-metric}
    A Riemannian manifold $(M^n,g)$ with boundary is called a conformally transversally anisotropic manifold, if there exists a Riemannian manifold $ (M_0^{n-1},g_0) $ such that $(\overline{M},g)\subset (\mathbb{R}\times M_0,c \cdot (e\oplus g_0))$, where $ (\mathbb{R},e) $ is the real space with canonical metric, and $ c(x)>0 $ is the conformal factor. Moreover, if $ c(x)=1 $, then the metric is said to be \textbf{transversally anisotropic} (TA for short). If in addition, $ (M_0,g_0) $ is a simple manifold, which means that $\partial M_0$ is strictly convex and that $exp_p$ is a diffeomorphism onto $\overline{M_0}$ for all $p\in M_0$, then we call $ M $ a \textbf{simple CTA manifold}.
\end{definition}
Now suppose $ (M,g)\subset (\mathbb{R}\times M_0,c \cdot (e\oplus g_0)) $ is a compact CTA manifold with smooth boundary $ \partial M $. The quasilinear anisotropic conductivity $ \gamma $, associated with $ g $, takes the form
\begin{equation}\label{def-form}
    \gamma (s,x,p)=a(s,x,p)A (x),
\end{equation}
where $ a\in C^{2,\alpha}(\mathbb{R}\times TM) $ is strictly positive on $ \overline{M} $, and that
\begin{align}\label{def:corresponding_metric}
    A=\sqrt{\det g}\cdot g^{-1}.
\end{align}\par
We may rewrite the anisotropic conductivity $ \gamma(s,x,p) $ as
\begin{equation*}
    \gamma(s,x,p)= \frac{a(s,x,p)}{c(x)^{\frac{n-2}{2}} }\left(c(x)^{\frac{n-2}{2}}A(x)\right).
\end{equation*}
With $ A $ in (\ref{def-form}) replaced by $ c^{\frac{n-2}{2}}A $ and $ a(s,x,p) $ by $ c(x)^{-\frac{n-2}{2}}a(s,x,p) $, the equation (\ref{eq:quasi_gen}) remains unchanged. Meanwhile, the corresponding Riemannian metric to $ c^{\frac{n-2}{2}}A $ is exactly $ e\oplus g_0 $. In this sense, we may assume that $ c\equiv 1 $ without loss of generality. 

The analogue of (\ref{eq:quasi_gen}) on $ M $ reads
\begin{equation}\label{eq:quasi_intrinsic}
    \left\{\begin{aligned}
    \nabla^*(a(u,x,\nabla u)\nabla u)  &= 0 \quad\text{in}~ M,\\
    u                     &= f \quad\text{on}~ \partial M,
    \end{aligned}\right.
\end{equation}
where $ \nabla^* $ and $ \nabla $ denote the divergence and gradient associated with $ g $ respectively. Lemma \ref{lmm:forward} yields a well-defined DN map $ N_{\gamma}  $ for each $ \gamma $ in (\ref{def-form}) on $ \partial M $. The inverse problem of concern is to retrieve the information of $ \gamma $ from the knowledge of $ N_{\gamma}  $. Our main result in this article is the following:

\begin{theorem}[]\label{thm:main}
    Let $ (M,\partial M, g) $ be a compact simple CTA manifold with smooth boundary $ \partial M $. Write $ \gamma_m(s,x,p):= a_m(s,x,p)A(x)$, for $m=1,2 $ as in (\ref{def-form}) and (\ref{def:corresponding_metric}). If $ N_{\gamma_1}=N_{\gamma_2} $, then for all $ (t,x)\in\mathbb{R}\times M $ and $ \beta\in\mathbb{N}^{1+n}$, it holds that
    \begin{equation}\label{deriv-equal}
        \partial_{(s,p)}^{\beta}  a_1(s,x,p)\Big|_{(s,p)=(t,0)}  = \partial_{(s,p)}^{\beta} a_2(s,x,p)\Big|_{(s,p)=(t,0)}.
    \end{equation}
    If in addition, $ a_m(s,x,p) $ is analytic in $ p $, then $ a_1=a_2 $.
\end{theorem}\par

\begin{remark}
    We use Einstein's summation convention throughout this article.  Unless otherwise specified, indices in $[1,n] \cap \mathbb{Z}$ are denoted by Roman letters, while those in $[0,n] \cap \mathbb{Z}$ are denoted by Greek letters. 
\end{remark}


For isotropic media, the classical framework for linear conductivity equations, introduced in \cite{sylvester1987global}, converts Calderón's problem to an inverse potential problem for Schr\"{o}dinger equations. The latter is achieved by establishing the injectivity of a certain integral transform of the potential through \textit{complex geometric optics} (CGO) solutions of the Schr\"{o}dinger equations. See also \cite{novikov1988multidimensional} for inverse potential problems of Schr\"{o}dinger equations. For nonlinear equations, the approach entails higher-order linearization, which reduces inverse conductivity problems for nonlinear equations to inverse potential problems for linearized equations (see, for example, \cite{feizmohammadi2020inverse,lassas2020partial,lassas2021inverse}). Generally, when the nonlinearity depends also on the electric fields, the relevant integral transforms involve multiple products of the gradient of harmonic functions, as shown in \cite{carstea2021calderon,carstea2023density,liimatainen2024calderon}.

Anisotropic media are more involved. On CTA manifolds, \cite{dossantosferreira2009limiting,dossantosferreira2016calderon,dossantosferreira2020linearized} utilized Gaussian beam quasi-modes to construct the principal terms of CGO solutions and applied Carleman estimates to analyze the remainder terms. This approach determines the conformal factor of the conductivity. Existing work on nonlinear equations in anisotropic media \cite{sun1997inverse,feizmohammadi2020inverse,carstea2021inverse} required various additional assumptions on the conductivity. In general, the CTA property required in Theorem \ref{thm:main} does not satisfy the conditions therein.


Our strategy to tackle Theorem \ref{thm:main} consists of two components. First, we follow the framework of \cite{carstea2021calderon} to linearize the problem, and to convert the problem of unique inversion of coefficients to the injectivity of some integral transforms.

Then, to prove (\ref{deriv-equal}) for each $ \beta $, we use the following methods to treat the relevant integral transforms.
\begin{itemize}
    \item For $ |\beta|=0 $, it is directly reduced to the inverse problem for linearized equations on CTA manifolds, solved in \cite{dossantosferreira2009limiting,dossantosferreira2016calderon,dossantosferreira2020linearized}.
    \item For $ |\beta| = 1 $, we first recover $ \partial_{(s,p)} ^{\beta}a(s,x,p)|_{(s,p)=(t,0)}   $ at the boundary, and then extend it to the interior. The proof at the boundary, inspired by \cite{sun1997inverse}, applies singular solutions constructed in \cite{alessandrini1990singular} to the integral transforms. The dependence of $ a(s,x,p) $ on $ p $ produces the gradients of harmonic functions in the integral transforms. We utilize Stokes' theorem and the boundary determination results to convert the integral transform involving gradients of harmonic functions to that involving only harmonic functions.
    \item For $ |\beta|\geq 2 $, we use the CGO solutions constructed in \cite{dossantosferreira2016calderon,feizmohammadi2020inverse} with the Gaussian beams concentrated near some non-tangential geodesics (see Definition \ref{def:non-tan-geodesic}). Inspired by \cite{lassas2021inverse}, we plug two pairs of such CGO solutions into relevant integral transforms. This reduces the integral transforms to a linear system, the coefficients of which are the tangent vectors of the geodesics at a point. The choice of tangent vectors ensures the uniqueness and  existence of the solution to this system.
\end{itemize}

\section{Preliminaries}

\subsection{Forward problems}
First, we need to show that the DN map for (\ref{eq:quasi_gen}) is well-defined. In fact, one can prove that the forward problem is well-posed near constant solutions, by adapting the proof of \cite[][Theorem B.1]{kian2023partial}.
\begin{lemma}\label{lmm:forward}
    Consider the boundary value problem  (\ref{eq:quasi_gen}) with $ (\gamma^{ij} ) $ satisfying the elliptic conditions (\ref{condition:elliptic}). Then for any $ \lambda\in\mathbb{R} $, there exists a positive number $ \delta(\lambda)$ such that 
      for any boundary value $f\in B^{2,\alpha}_{\partial M} (\lambda,\delta(\lambda))$ the problem (\ref{eq:quasi_gen}) admits a unique solution $ u^f $ near the constant solution $ \lambda $.
\end{lemma}

Then one can verify the gauge-invariance of the DN map under the gauge transform $ H_{\Phi}  $ in \eqref{def:gauge-transform}. 
\begin{lemma}\label{lmm:gauge-invariance} Let $\gamma$ be a conductivity in (\ref{eq:quasi_gen}).
  If $ \Phi\in\text{\rm Diff}(\overline{M}) $ is a $ C^{3,\alpha}$-diffeomorphism  satisfying $ \Phi_{\partial M} =\mathrm{id}_{\partial M}  $, then $ N_{H_{\Phi}\gamma }=N_{\gamma}.$
\end{lemma}
\begin{proof}[Proof of Lemma \ref{lmm:gauge-invariance}]
    Let $ u^f\in B^{2,\alpha}_{\partial M} (\lambda,\delta(\lambda)) $ be some solution of (\ref{eq:quasi_gen}), and let $ \tilde{u}=u^f\circ(\Phi^{-1}) $ with $ \Phi\in \text{\rm Diff}(\overline{M}) $ a $ C^{3,\alpha} $ diffeomorphism satisfying $ \Phi_{\partial M} =id_{\partial M}  $. The function $ \tilde{u} $ will thus be the solution of the following equation:
    \begin{equation}
        \left\{\begin{aligned}
        \partial_i(\gamma^{ij} (\tilde{u}\circ\Phi,x,\nabla (\tilde{u}\circ\Phi))\partial_j (\tilde{u}\circ\Phi) )&= 0 \quad\text{in}~M,\\
        \tilde{u}      &= f \quad\text{on}~ \partial M.
        \end{aligned}\right.
    \end{equation}
    This means that for any $ \psi=\varphi\circ\Phi\in C_c^{\infty}(M) $, it holds that
    \begin{equation*}
        \int_M \gamma^{ij} (\tilde{u}\circ\Phi(x),x,\nabla (\tilde{u}\circ\Phi)(x))\partial_j (\tilde{u}\circ\Phi)(x)\partial_i(\varphi\circ\Phi)(x)dx=0.
    \end{equation*}\par
    Let $ x=\Phi^{-1}(y) $, the integral above becomes
    \begin{equation*}
        \int_M \frac{(D\Phi_{\Phi^{-1}(y)}\nabla \varphi)_i(y)\gamma^{ij} (\tilde{u}(y),\Phi^{-1}(y),(D\Phi_{\Phi^{-1}(y)}\nabla \tilde{u})(y)) (D\Phi_{\Phi^{-1}(y)}\nabla \tilde{u})_j(y)} {|\det D\Phi|_{\Phi^{-1}(y)}}dy =0,
    \end{equation*}
    which can be exactly rewritten as the following:
    \begin{equation*}
        \int_M \partial_k\varphi(y)(H_\phi\gamma)^{kl}(\tilde{u}(y),y,\nabla\tilde{u}(y))\partial_l \tilde{u}(y)dy=0,\quad\forall \varphi\in C_c^{\infty}(M).
    \end{equation*}\par
    Therefore, we know that $ \tilde{u}=S^{H_\Phi \gamma}(f) $, and that
    \begin{multline*}
        N_{H_\Phi\gamma}(f)(y)= (H_\phi\gamma)^{kl}(\tilde{u}(y),y,\nabla\tilde{u}(y))\partial_{y^k}  \tilde{u}(y)\iota^*(\partial_{y^l}  \llcorner dy^1\wedge dy^2\wedge\cdots\wedge dy^n) \\
        \xlongequal[u^f=\tilde{u}\circ\Phi]{\text{\rm Let }y=\Phi(x)}\gamma^{ij}(u(x),x,\nabla u(x))\partial_{x^i}u(x)\iota^*(\partial_{x^j}  \llcorner dx^1\wedge dx^2\wedge\cdots\wedge dx^n)=N_{\gamma}(f)(x)\\
        \xlongequal{\Phi_{\partial M}=id_{\partial M}} N_{\gamma}(f)(y),\quad\forall y\in\partial M,
    \end{multline*}
which completes the proof.
\end{proof}
In fact, we may view the conductivity $ (\gamma^{ij}(s,\cdot,p)) $ at fixed $ (s,p)\in\mathbb{R}\times\mathbb{R}^n $ as a Riemannian metric $ (g_{ij}(\cdot)) $. The gauge transform of the conductivity is indeed the pullback of the metric $ (g_{ij} ) $. This is clear as we have the following relationship between the equations:
\begin{equation*}
    \partial_i(\gamma^{ij}\partial_j u)=(\det g)\Delta_g(u),~\text{namely}~\partial_i(\gamma^{ij}\partial_j u)=0\Leftrightarrow \Delta_g(u)=0,
\end{equation*}
and that between the corresponding DN maps: 
\begin{equation*}
    N_g u:=\iota^*(\text{\rm grad}^{g} u \llcorner d\text{\rm Vol}_g)=\iota^*(\partial_{x^i}u\gamma^{ij}\partial_{x^j}  \llcorner dx^1\wedge dx^2\wedge\cdots\wedge dx^n)=N_\gamma u.
\end{equation*}

\subsection{The Calder\'on problem for linear conductivities}

We list some results on inversion of linear conductivities, 
which are useful for nonlinear conductivity problems.

First, \cite{dossantosferreira2009limiting, dossantosferreira2016calderon} showed that the conformal factor of a linear CTA conductivity is uniquely determined by the DN map.
\begin{proposition}\label{Prop-0th}
    Let $ (M,\partial M, g) $ and $ A $ be as in Theorem \ref{thm:main} and (\ref{def:corresponding_metric}).   Consider the linear conductivities $ \gamma_m:=c_m A $ with positive factors $ c_m\in C^{\infty} (M) $ for $m=1, 2$. If $ N_{\gamma_1}=N_{\gamma_2} $ holds, then we have $c_1(x)=c_2(x)$ for any $x \in M$.
\end{proposition}\par
The core of recovering conformal factors in Proposition \ref{Prop-0th} is to show the products of harmonic functions are dense in $ \mathcal{D}'(M) $, which was proved in \cite[][Subsection 6.1]{dossantosferreira2009limiting}:
\begin{lemma}\label{lmm:density-cta}
    Let $ (M,g) $ be as in Theorem \ref{thm:main} and $ f\in L^{\infty}(M)  $. Denote by $\mathcal{H}_g(M)$ the collection of harmonic functions on $(M, g)$.  If there holds   \begin{displaymath}
    	\int_M fu_1u_2=0,\quad \forall (u_1, u_2) \in \mathcal{H}_{c_1g}(M) \times \mathcal{H}_{c_2g}(M),
    \end{displaymath}
    then $ f\equiv 0 $.
\end{lemma}\par

Retrieving the second derivative requires special solutions with prescribed singularity only at a single point, which was constructed in \cite[Theorem 1.1]{alessandrini1990singular}.  
\begin{proposition}[]\label{Prop-Green}
    Let $ \Omega\subset\mathbb{R}^n $ be a bounded domain, $ x_0\in \Omega $ and $ (A^{ij}) \in C^{2,\alpha}(\Omega,\mathrm{Mat}_{n}(\mathbb{R}) ) $. Suppose that $ x_0=0$ and $A^{ij}(x_0)=\delta^{ij} $. There exists a solution of $ \partial_i(A^{ij} \partial_j w)=0 $ with prescribed singularity at the point $ x_0=0 $. Specifically, such a solution takes the form
    \begin{equation}\label{singular_solution}
        w(x)=H_n(x)+\omega_{n} (x),
    \end{equation}
    where $ H_n =|x|^{2-n}$ is the Green function,
    and $ \omega_{n}\in C^{3,\alpha}(\overline{\Omega}) $ satisfies the order estimate
    \begin{equation}\label{singular_solution-estimate}
        |\omega_{n}(x)|+|x||\nabla\omega_{n}(x)|\leq C_{\beta,\overline{\Omega},(A^{ij} )}   |x|^{2-n+\beta},
    \end{equation}
    with $ \beta\in(0,1) $ and $C_{\beta,\overline{\Omega},(A^{ij} )} <+\infty $.
\end{proposition}\par
The reconstruction of higher-order derivatives involves non-tangential geodesics on $ (M_0,g_0) $ introduced in \cite[Definition 1.2]{dossantosferreira2020linearized}.
\begin{definition}[Non-tangential geodesics]\label{def:non-tan-geodesic}
    A geodesic $ \gamma:[-T,T]\to M_0 $ is called non-tangential, if there holds that \begin{itemize} \item the interior of $\gamma$ lies in the interior of $M_0$, i.e., $ \{\gamma(x) : x\in(-T,T)\} \subset M_0^{int} $; \item   the endpoints of $\gamma$ are at the boundary of $M_0$, i.e., $ \gamma(\pm T)\in \partial M $; \item $\gamma$ is not tangent to the boundary of $M_0$, i.e., $ \dot{\gamma}(\pm T)\notin T_{\gamma(\pm T)}(\partial M_0) $. \end{itemize}
\end{definition}\par

In particular, we need to estimate some products of the gradients of harmonic functions, where CGO solutions are effective. In this circumstance, a key component of the suitable CGO solutions is the following Gaussian beam quasi-modes from \cite[Section 4.1]{feizmohammadi2020inverse} and \cite[Proposition 5.2]{lassas2021inverse}.
\begin{proposition}[Gaussian beam quasi-modes]\label{Prop-GBQ}
    Suppose $ (M_0,g_0) $ is a compact Riemannian manifold with smooth boundary $ \partial M_0 $ of dimension $n-1$. Let $ \gamma:[-T,T]\to M_0 $ be a non-tangential geodesic. For any $ k,K\in\mathbb{N} $ and $ \delta>0 $, there is a family of functions $$ \left\{v_s \in C^{\infty} (M_0) : s=\tau+i\lambda, \tau\geq 1 ,\lambda\in\mathbb{C},|\lambda|<1 \right\} $$ satisfying the estimates
    \begin{equation*} \left\{
        \begin{aligned}
            \left\|(-\Delta_{g_0}-(\tau+i\lambda)^2)v_{\tau+i\lambda}\right\|_{H^k(M_0)} &=O(\tau^{-K} ) \\
            \left\|v_{\tau+i\lambda} \right\|_{L^4( M_0)} &=O(\tau^{-\frac{n-2}{8}})\\
            \left\|v_{\tau+i\lambda} \right\|_{L^4(\partial M_0)} &=O(\tau^{-\frac{n-2}{8}})
        \end{aligned}\right.  \quad \mbox{ as $ \tau\to\infty $.}
    \end{equation*}
    Specifically, each $ v_s $ reads
    \begin{equation}
        v_s(x')=e^{i(\tau+i\lambda)\psi(x')}b_{ \tau+i\lambda}(x'),
    \end{equation}
    where  
    \begin{itemize}
    	\item    the phase function $ \psi(x') $ satisfies
    	\begin{equation*}
    		\psi(\gamma(t))=t,\quad \nabla\psi(\gamma(t))=\dot{\gamma} (t),\quad \text{\rm Im}((\psi\circ\gamma)''(t))\geq 0,\quad \text{\rm Im}(\nabla^2\psi(\gamma(t)))|_{\dot{\gamma}(t)^{\bot} } >0,
    	\end{equation*}
    	for $ t $ close to $ t_0 $;
    	
    	\item  the amplitude $ b_{ \tau+i\lambda}(x') $  is supported within a $ \delta $-tubular neighborhood of the geodesic $ \gamma $, and admits an expansion along $\gamma$, $$ b_{ \tau+i\lambda}(\gamma(t))=\tau^{\frac{n-2}{4}}(b_{(0)}(\gamma(t))+O(\tau^{-1})), $$ with $ b_{(0)}(\gamma(t)) $  non-vanishing. 
    \end{itemize}

\end{proposition}\par

The desired CGO solutions, constructed in \cite[(44)]{feizmohammadi2020inverse} and \cite[Proposition 5.3]{lassas2021inverse}, are the following.

\begin{proposition}[CGO solutions]\label{Prop-CGO}
    Let $(M,g) $ be as in Theorem \ref{thm:main} and $ \gamma:[-T,T]\to M_0 $ a non-tangential geodesic. Then for any $ \tau\geq 1,~\delta > 0,~\lambda\in\mathbb{C},|\lambda|<1 $ and $ k,K\in\mathbb{N} $, the equation $ \Delta_{cg} u=0 $ admits a CGO solution of the form:
    \begin{equation*}
        u_{\tau+i\lambda} = e^{-(\tau+i\lambda)x^1} (v_{ \tau+i\lambda}(x')+r_{\tau+i\lambda})
    \end{equation*} 
where $ v_{ \tau+i\lambda}(x')=e^{i(\tau+i\lambda)\psi(x')}b_{ \tau+i\lambda}(x')  $ is the Gaussian beam quasi-mode given as in Proposition \ref{Prop-GBQ}.
More precisely, for any $K \in \mathbb{Z}_+$ one can find
\begin{equation}\label{eq:gbq-wkb}
    b_{\tau+i\lambda}=\tau^{\frac{n-2}{4}}\sum_{k=0}^{N}(\tau+i\lambda)^{-k}b_{(-k)}
\end{equation}
such that the remainder term $ r_{\tau+i\lambda} $ obeys 
\begin{equation}\label{eqn : error estimates for CGO}
    \begin{aligned}
            \|r_{\tau+i\lambda} \|_{H^k(M)} &=O(\tau^{-K}),\quad \mbox{    as $ |\tau|\to\infty $   }.
    \end{aligned}
\end{equation}

\end{proposition}

\subsection{Linearization of Nonlinear Equations}
We perform $L$-fold linearization on \eqref{eq:quasi_intrinsic} for $ L \in \mathbb{Z}_+$.  Specifically, we take $ (h_1, \cdots, h_L) \in (C_c^\infty (\partial M))^L,$ and  $ \epsilon:=(\epsilon^1,\cdots,\epsilon^L)\in \mathbb{R}^L $ close to $ 0 $, and prescribe for \eqref{eq:quasi_intrinsic} the boundary value
\begin{equation*}
    f_\epsilon:=f_\epsilon(t,h_1,\cdots,h_L)=t+\sum_{k=1}^{L}\epsilon^k h_k\in B_{\partial M}^{2,\alpha}(t,\delta(t)) 
\end{equation*}
lying in (\ref{eq:def-domain}). The $ L $-th order Fréchet derivative of $ N_{\gamma} $ at the constant function $ t $ reads 
\begin{displaymath}
    (D^L N_{\gamma})_t(h_1,\cdots,h_L):=\partial_{\epsilon^1}\cdots\partial_{\epsilon^L}N_{\gamma} \left(t+\sum_{k=1}^{L}\epsilon^k h_k\right)\bigg|_{\epsilon = 0}.
\end{displaymath}
Analogously, if we let $ u^{f_\epsilon} $ be the unique solution, near $t$, to (\ref{eq:quasi_gen}) with boundary value $ f_\epsilon $, we write the Fréchet multi-derivative of $ u^{f_\epsilon} $ at $ \epsilon=0 $ as
\begin{displaymath}
    w_t^{h_1,\cdots,h_L}:=\partial_{\epsilon^1}\cdots\partial_{\epsilon^L}u^{f_\epsilon}\big|_{\epsilon =0}.
\end{displaymath}\par

When $ L=1 $, we have $ f_\epsilon=t+\epsilon^1 h_1 $. Now we make the Taylor expansion of $ u^{f_\epsilon} $  at $ \epsilon^1=0 $, which is of the form
\begin{equation*}
    u^{f_\epsilon} =t+\epsilon^1 w^{h_1}+o(|\epsilon|).
\end{equation*}
To compute the linearization of  (\ref{eq:quasi_intrinsic}), we need the Taylor expansion of the conductivity
\begin{equation*}
    a(s,x,p)=a(t,x,0)+O(|s-t|+|p|).
\end{equation*}
Combining the two formulae above leads to
\begin{equation*}
    a\left(u^{f_\epsilon},x,\nabla u^{f_\epsilon}\right)\nabla u^{f_\epsilon}=\epsilon^1 a(t,x,0)\nabla w^{h_1} +o(\epsilon^1)
\end{equation*}
and thus the linearization of (\ref{eq:quasi_intrinsic}) is
\begin{equation}\label{eq:1st-order-int} \left\{\begin{aligned}
    \nabla^*(a(t,x,0)\nabla w^{h_1})&=0\\ w^{h_1}|_{\partial M} &= h^1. \end{aligned}\right.
\end{equation}
Moreover, the DN map of (\ref{eq:1st-order-int}) is given by 
\begin{equation*}
    (DN_{\gamma})(h^1)= a(t,x,0)\nu\cdot\nabla w^{h_1}.
\end{equation*}\par

When $L=2$, we have $ f_\epsilon=t+\epsilon^1 h_1+\epsilon^2 h_2 $. 
Since $ u^{t} \equiv t $ for $t \in \mathbb{R}$, we first compute
\begin{align*}
    \partial_{\epsilon^1}\partial_{\epsilon^2}\left(a\left(u^{f_\epsilon},x,\nabla u^{f_\epsilon}\right)\nabla u^{f_\epsilon}\right) \Big|_{\epsilon = 0} =& \partial_{\epsilon^1}a\left(u^{f_\epsilon},x,\nabla u^{f_\epsilon}\right)\nabla \left(\partial_{\epsilon^2}u^{f_\epsilon}\right) \Big|_{\epsilon = 0}\\
    &+\partial_{\epsilon^2}a\left(u^{f_\epsilon},x,\nabla u^{f_\epsilon}\right)\nabla \left(\partial_{\epsilon^1}u^{f_\epsilon}\right) \Big|_{\epsilon = 0}\\
    &+a(t,x,0)\nabla \left(\partial_{\epsilon^1}\partial_{\epsilon^2}u^{f_\epsilon}\right) \Big|_{\epsilon = 0}.
\end{align*}
Note that $\partial_{\epsilon^j}u^{f_\epsilon}$ solves \eqref{eq:1st-order-int} with index $1$ replaced by $j$.
Then differentiating (\ref{eq:quasi_intrinsic}) twice in $ \epsilon^1$ and $\epsilon^2 $ respectively at $\epsilon = 0$ yields
\begin{multline}\label{eq:2nd-order-orig}
    \nabla^*\big(w^{h_1}\partial_s a(t,x,0)\nabla w^{h_2}+w^{h_2}\partial_s a(t,x,0)\nabla w^{h_1}+\partial_l w^{h_1}\partial_{p^l} a(t,x,0) \nabla w^{h_2}\\
    +\partial_l w^{h_2}\partial_{p^l} a(t,x,0) \nabla w^{h_1}+a(t,x,0)\nabla w^{h_1,h_2}\big)=0.
\end{multline}

To generalize this calculation to higher-orders, we  introduce the following shorthand notations. For any $ 0\leq\alpha \leq n$, we define the operator $ P_\alpha $ by
\begin{equation}\label{def:P}
    P_\alpha u(x):=\begin{cases}
        u(x) & ,\alpha=0, \\
        \partial_{l} u(x) & ,1\leq \alpha=l\leq n.
    \end{cases}
\end{equation}
We let the tensor $ (T^{\alpha_1,\cdots,\alpha_L})_{0\leq \alpha_1,\cdots,\alpha_L\leq n}  $ denote the higher order derivatives of $a(s,x,p)$ 
\begin{equation}\label{def:T}
    T^{\alpha_1,\cdots,\alpha_L}(x) := Q_{\alpha_1}\cdots Q_{\alpha_L} a(s,x,p)|_{(s,p)=0},  
\end{equation}
which consists of first derivatives
\begin{equation*}
    Q_\alpha u(s,x,p):=\begin{cases}
        \partial_s u(s,x,p), & \alpha=0, \\
        \partial_{p^l} u(s,x,p), & 1\leq \alpha=l\leq n.
    \end{cases}
\end{equation*}
For convenience in doing summation, we also write
\begin{equation}\label{def:T-phi}
    \begin{aligned}
        T^J(u_1,\cdots,u_J)(x)&:=T^{\alpha_1,\cdots,\alpha_J}(x)P_{\alpha_1}u_1(x) P_{\alpha_2}u_2(x) \cdots P_{\alpha_J}u_J(x), \\	
        \phi^{L}(h_{1},\cdots,h_{L+1})(x)&:=\sum_{\substack{1\leq J\leq L-1\\1\leq k_1<\cdots<k_J\leq L}} T^J(w^{h_1,\cdots,h_{k_1} },\cdots,w^{h_{k_{J-1}+1},\cdots,h_{k_J}})\nabla w^{h_{k_J+1},\cdots,h_{L+1}}.
    \end{aligned}
\end{equation}

With these notations,  (\ref{eq:2nd-order-orig}) is rewritten as 
\begin{equation}\label{eq:2nd-order}  
	\left\{
	\begin{aligned}
    \nabla^*\left(a(t,x,0) \nabla w^{h_1,h_2}+\sum\limits_{\sigma\in\pi(2)}T^\alpha P_\alpha w^{h_{\sigma(1)}}\nabla w^{h_{\sigma(2)}}\right)&=0\\ w^{h_1,h_2}|_{\partial M}&=0, \end{aligned}\right.
\end{equation}
where $ \pi(k) $ denotes the permutation group of order $ k $. Meanwhile, the second-order linearized DN map is expressed by 
\begin{equation}\label{eq:2nd-order-DN-0}
    (D^{2}N_{\gamma} )_t(h_1,h_2)=\nu\cdot\left(a(t,x,0) \nabla w^{h_1,h_2}+\sum\limits_{\sigma\in\pi(2)}T^\alpha P_\alpha w^{h_{\sigma(1)}}\nabla w^{h_{\sigma(2)}}\right).
\end{equation}
It turns out that the value of $ t $ actually does not affect the proof. Without loss of generality, we can fix $ t=0 $.\par

Through an analogous process of differentiation, the higher-order linearization of (\ref{eq:quasi_intrinsic}) is given by
\begin{equation}\label{eq:higher-order-0}
    \left\{\begin{aligned}
        \nabla^* (a(t,x,0) \nabla w ^{h_1,\cdots,h_{L+1}})&=-\sum\limits_{\sigma\in\pi(L+1)}\nabla^*\big(\phi^{L}(h_{\sigma(1)},\cdots,h_{\sigma(L+1)}) &\\
        &\qquad +T ^L( w^{h_{\sigma(1)}},\cdots,  w^{h_{\sigma(L)}})\nabla w^{h_{\sigma(L+1)}}\big)&\quad\text{\rm in}~M,\\
        w ^{h_1,\cdots,h_{L+1}} & =0 & \quad\text{\rm on}~\partial M,
    \end{aligned}\right.
\end{equation}
where $ L\geq 2 $, with the corresponding DN map
\begin{equation}\label{eq:higher-order-0-DN}
    \begin{aligned}
        \lefteqn{(D^{L+1}N_{\gamma} )_t(h_1,h_2,\cdots,h_{L+1})}\\
        &=\nu\cdot\bigg(a(t,x,0) \nabla w ^{h_1,\cdots,h_{L+1}}+\sum\limits_{\sigma\in\pi(L+1)}\phi^{L}(h_{\sigma(1)},\cdots,h_{\sigma(L+1)})\\
        & \qquad +\sum\limits_{\sigma\in\pi(L+1)}T ^L( w^{h_{\sigma(1)}},\cdots,  w^{h_{\sigma(L)}})\nabla w^{h_{\sigma(L+1)}} \bigg).
    \end{aligned}
\end{equation}

\section{Recover the Lower-Order Derivatives}
In this section, we recover the zeroth and the first derivatives of $ a(s,x,p) $ in $s$ and $p$ at $(s, p) = (t, 0)$. 

\subsection{First-Order Linearization: Recover the Conformal Factor}
Let $ h\in C^{\infty} (\partial M) $ and $ w_{m}^h(x) $ be the solution to (\ref{eq:1st-order-int}) with coefficient $a_m(t,x,0)$ for $m=1,2$. If the DN maps of (\ref{eq:quasi_intrinsic}) with $ a_1(s,x,p) $ and $ a_2(s,x,p) $ coincide, the DN maps for the first-order linearized equations (\ref{eq:1st-order-int}) with $ a_1(t,x,0) $ and $ a_2(t,x,0) $ must agree as well. 

Employing Proposition \ref{Prop-0th}, we can uniquely recover the conformal factor; namely, we can write
\begin{equation}\label{eq:1st-order-coeff}
    a(t,x,0):=a_1(t,x,0)=a_2(t,x,0).
\end{equation}


Moreover,   uniqueness of the solution to (\ref{eq:1st-order-int}) also implies that we can define
\begin{equation}\label{eq:1st-sol}
    w^h:=w_{1}^{h}=w_{2}^{h}   \quad \mbox{in $M$}.
\end{equation}\par

\subsection{Second-Order Linearization: Recover the First Derivatives}
When $ L=2 $, we differentiate (\ref{eq:quasi_intrinsic}) by $ \partial_{\epsilon^1}\partial_{\epsilon^2}  $ at $ \epsilon=0 $, and obtain the equation (\ref{eq:2nd-order}) with $ a_1 $ and $ a_2 $. For $ h_1,h_2\in C^{\infty} (\partial M) $, denote by $ w_{m}^{h_1,h_2} $ the solution to (\ref{eq:2nd-order}) with $a_m(s,x,p),m=1,2$ and consider the function 
$$w_0^{h_1,h_2}:=w_1^{h_1,h_2}-w_2^{h_1,h_2},$$
the coefficients $$T_0^\alpha(x):=T_1^\alpha(x)-T_2^\alpha(x),$$ and the operator $$N := N_{\gamma_1} - N_{\gamma_2}.$$
We plug (\ref{eq:1st-order-coeff}) and (\ref{eq:1st-sol}) into (\ref{eq:2nd-order}) and (\ref{eq:2nd-order-DN-0}) for $ m=1,2 $ respectively. Their difference is
\begin{equation}\label{eq:2nd-order-1}
    \begin{cases}
        \nabla^*\left(a(t,x,0)\nabla w_0^{h_1,h_2}\right)=-\sum\limits_{\sigma\in\pi(2)}\nabla^*\left(T_0^\alpha P_\alpha w^{h_{\sigma(1)}}\nabla w^{h_{\sigma(2)}}\right)&\quad\text{\rm in}~M,\\
        w_0^{h_1,h_2}=0&\quad\text{\rm on}~\partial M,
    \end{cases}
\end{equation}
and 
\begin{equation}\label{eq:2nd-order-DN}
    0=\left(D^2 N\right)_t (h_1,h_2)=\nu\cdot \left(a(t,x,0)\nabla w_0^{h_1,h_2}+\sum\limits_{\sigma\in\pi(2)}T_0^\alpha P_\alpha w^{h_{\sigma(1)}}\nabla w^{h_{\sigma(2)}}\right) .
\end{equation}\par
Let us temporarily write
$$ X:= \left(a(t,x,0)\nabla w_0^{h_1,h_2}+\sum\limits_{\sigma\in\pi(2)}T_0^\alpha P_\alpha w^{h_{\sigma(1)}}\nabla w^{h_{\sigma(2)}}\right). $$
Then, equation (\ref{eq:2nd-order-1}) is rewritten as $ \nabla^*X=0 $. Applying Stokes' formula to the weak form of (\ref{eq:2nd-order-DN}), we have that for any $ h_3\in C^\infty(\partial M) $ there holds
\begin{equation}\label{eq:2nd-order-IBP}
    \begin{split}
    0&=\int_{\partial M} h_3 \nu\cdot X =\int_M w^{h_3} \nabla^*X+\langle\nabla w^{h_3},X \rangle=\int_M w^{h_3} \langle\nabla w^{h_3},X \rangle\\
    &\xlongequal{}\int_M \langle\nabla w^{h_3},a(t,x,0)\nabla w_0^{h_1,h_2}\rangle+\int_M \langle\nabla w^{h_3},X-a(t,x,0)\nabla w_0^{h_1,h_2} \rangle.
    \end{split}
\end{equation} 
Applying Stokes' formula to the first integral with the boundary condition in (\ref{eq:2nd-order-1}) gives
\begin{equation*}
    \int_M \langle\nabla w^{h_3},a(t,x,0)\nabla w_0^{h_1,h_2}\rangle=-\int_M \nabla^*(a(t,x,0)\nabla w^{h_3})w_0^{h_1,h_2}=0.
\end{equation*}
Therefore, the second integral in the last line of (\ref{eq:2nd-order-IBP}) must vanish; namely
\[
    \int_M \langle\nabla w^{h_3},X-a(t,x,0)\nabla w_0^{h_1,h_2} \rangle=0.
\]
It follows that for any $ h_1,h_2,h_3\in C^\infty(\partial M) $
\begin{equation}\label{eq:2nd-order-2}
    \sum_{\sigma\in\pi(2)}\int_M T_0^\alpha(x) P_\alpha w^{h_{\sigma(1)}}\langle\nabla w^{h_{\sigma(2)}},\nabla w^{h_3}\rangle=0.
\end{equation}

We first show $ T_0^0\equiv 0 $. Taking $ h_1\equiv 1 $ in (\ref{eq:2nd-order-2}) yields
\begin{displaymath}
    \int_M T_0^0(x) \langle\nabla w^{h_{2}},\nabla w^{h_3}\rangle=0,
\end{displaymath}
and it suffices to prove the following proposition.

\begin{proposition}\label{Prop-basic}
    Let $ (M,g) $ be as in Theorem \ref{thm:main}. Suppose $ \phi $ is a differentiable function on $ \overline{M} $. If any two solutions $ w^{h_1} ,w^{h_2}   $ of the equation (\ref{eq:1st-order-int}) with boundary values $ h_1,h_2 $ respectively satisfy the identity
    \begin{equation}\label{eq:2nd-order-8}
        \int_M \phi \langle \nabla w^{h_1} ,\nabla w^{h_2} \rangle =0,
    \end{equation}
    then the function $ \phi $ must vanish in $ M $.
\end{proposition}\par

    
Firstly, we demonstrate that $ \phi $ vanishes on the boundary. 
\begin{proof}[Proof of Proposition \ref{Prop-basic} on the boundary]
    We prove the proposition by making order estimates for (\ref{eq:2nd-order-8}) with a series of singular solutions (\ref{singular_solution}), which is inspired by \cite[Lemma 4.7]{sun1997inverse}. 
   
   Fix an arbitrary boundary point $ x_0\in\partial M $. Since $ (M,g)=(\mathbb{R}\times M_0,c \cdot (e\oplus g_0)) $ is a simple CTA manifold as defined in Definition \ref{def:cta-metric}, $ M $ is diffeomorphic to a Euclidean domain in $ \mathbb{R}^{n}  $, and we define $ (A^{ij} )=\sqrt{\det g}(g_{ij} )^{-1} $ as in (\ref{def:corresponding_metric}). Moreover, we denote $ \widetilde{A}^{ij} (x):=a(t,x,0) A^{ij} (x) $, so that the equation (\ref{eq:1st-order-int}) is equivalent to
   \begin{equation}\label{eq:prop5}
    \partial_i(\widetilde{A}^{ij}\partial_j w)=0.
   \end{equation}
    
    Without loss of generality, we can move $ x_0 $ to the origin and choose suitable coordinates such that $ \partial M \cap B_{2\rho}(x_0) \subset \mathbb{R}^{n-1}\times\{0\} $ for some $ \rho $, where $$ B_R(y):= \left\{x\in\mathbb{R}^{n}~|~\mathrm{dist}_{\mathbb{R}^n}(x,y)<R  \right\}. $$
    Moreover, we may require that $ \widetilde{A}^{ij}(x_0)=\widetilde{A}^{ij} (0)=\delta^{ij} $ and the outer normal vector $ \nu(x) $ at $ x_0 $ is $ \nu(x_0)=(0,\cdots,0,-1) $. Then we define $ \widetilde{M} $ by gluing to $ M $ a sufficiently small semi-ball $$B^-_{2\rho}(x_0):=B_{2\rho}(x_0)\cap\{x^n-y^n<0\} $$ with metric $ \tilde{g} $ a smooth extension of $ g $.\par
    Since $ \partial M $ is smooth, we can choose $ \rho $ sufficiently small such that $ B_{2\rho}(x_0)\cap M=B_{2\rho}^+(x_0)\cap M $, where $ B_{\rho}^+(y):=B_{\rho}(y)\cap\{x^n-y^n>0\} $. We consider a family of points 
    $$ x_\tau:=\tau \nu(x_0)=(0,\cdots,0,-\tau),~\mbox{for $ 0<\tau<\rho $},$$
    and the restriction of $ B_{\rho}(x_\tau) $ at $ \partial M $
    $$ B_{\rho}^{n-1}(x_\tau):=B_{\rho}(x_\tau)\cap \partial M =B_{\rho}(x_\tau)\cap (\mathbb{R}^{n-1}\times\{0\}). $$
   These are as depicted below.\par

    
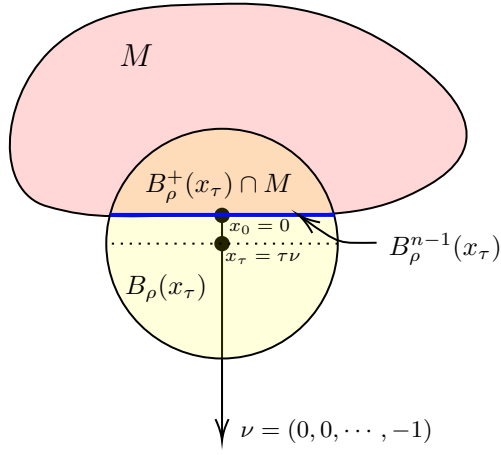
\begin{center}
    \begin{tikzpicture}[x=0.75pt,y=0.75pt,yscale=-1,xscale=1]
        \draw  [fill={rgb, 255:red, 255; green, 0; blue, 0 }  ,fill opacity=0.16 ][line width=0.75]  (235.48,194.76) .. controls (210.25,187.45) and (211.7,142.82) .. (236.73,117.21) .. controls (261.77,91.61) and (300.67,93.8) .. (321.85,93.07) .. controls (343.03,92.34) and (378.43,93.07) .. (414.48,118.68) .. controls (450.53,144.28) and (455.36,167.2) .. (426.51,184.03) .. controls (397.67,200.86) and (376.04,199.39) .. (323.17,199.39) .. controls (270.29,199.39) and (260.72,202.08) .. (235.48,194.76) -- cycle ;
        \draw    (323.17,199.39) -- (323.17,310.81) ;
        \draw [shift={(323.17,312.81)}, rotate = 270] [color={rgb, 255:red, 0; green, 0; blue, 0 }  ][line width=0.75]    (10.93,-3.29) .. controls (6.95,-1.4) and (3.31,-0.3) .. (0,0) .. controls (3.31,0.3) and (6.95,1.4) .. (10.93,3.29)   ;
        \draw [line width=0.75]    (323.17,196.59) -- (323.17,199.39) ;
        \draw [shift={(323.17,199.39)}, rotate = 90] [color={rgb, 255:red, 0; green, 0; blue, 0 }  ][fill={rgb, 255:red, 0; green, 0; blue, 0 }  ][line width=0.75]      (0, 0) circle [x radius= 3.35, y radius= 3.35]   ;
        \draw [line width=0.75]    (323.17,208) -- (323.17,213.62) ;
        \draw [shift={(323.17,213.62)}, rotate = 90] [color={rgb, 255:red, 0; green, 0; blue, 0 }  ][fill={rgb, 255:red, 0; green, 0; blue, 0 }  ][line width=0.75]      (0, 0) circle [x radius= 3.35, y radius= 3.35]   ;
        \draw  [fill={rgb, 255:red, 255; green, 255; blue, 0 }  ,fill opacity=0.14 ][line width=0.75]  (265.39,213.62) .. controls (265.39,181.71) and (291.26,155.84) .. (323.17,155.84) .. controls (355.08,155.84) and (380.94,181.71) .. (380.94,213.62) .. controls (380.94,245.53) and (355.08,271.4) .. (323.17,271.4) .. controls (291.26,271.4) and (265.39,245.53) .. (265.39,213.62) -- cycle ;
        \draw [line width=0.75]  [dash pattern={on 0.84pt off 2.51pt}]  (380.94,213.62) -- (265.39,213.62) ;
        \draw [color={rgb, 255:red, 0; green, 12; blue, 252 }  ,draw opacity=1 ][fill={rgb, 255:red, 0; green, 0; blue, 0 }  ,fill opacity=1 ][line width=1.5]    (267.33,199.17) -- (379.33,199.17) ;
        \draw    (400.33,213.17) .. controls (380.83,213.17) and (383.2,215.07) .. (362.94,200.33) ;
        \draw [shift={(361.33,199.17)}, rotate = 36.03] [color={rgb, 255:red, 0; green, 0; blue, 0 }  ][line width=0.75]    (10.93,-3.29) .. controls (6.95,-1.4) and (3.31,-0.3) .. (0,0) .. controls (3.31,0.3) and (6.95,1.4) .. (10.93,3.29)   ;

        \draw (270.56,111.89) node [anchor=north west][inner sep=0.75pt]    {$M$};
        \draw (325.17,199.99) node [anchor=north west][inner sep=0.75pt]  [font=\fontsize{0.59em}{0.71em}\selectfont]  {$x_{0 } = 0$};
        \draw (330.94,300.2) node [anchor=north west][inner sep=0.75pt]  [font=\scriptsize]  {$\nu =( 0,0,\cdots ,-1)$};
        \draw (323.17,217.02) node [anchor=north west][inner sep=0.75pt]  [font=\fontsize{0.59em}{0.71em}\selectfont]  {$x_{\tau } =\tau \nu $};
        \draw (273.32,227.61) node [anchor=north west][inner sep=0.75pt]  [font=\footnotesize]  {$B_{\rho}( x_{\tau })$};
        \draw (283.33,175.57) node [anchor=north west][inner sep=0.75pt]  [font=\footnotesize]  {$B_{\rho}^{+}( x_{\tau }) \cap M$};
        \draw (398.33,206.57) node [anchor=north west][inner sep=0.75pt]  [font=\footnotesize]  {$ \begin{array}{l}
        B_{\rho }^{n-1}( x_{\tau })
        \end{array}$};
    \end{tikzpicture}
    \end{center}

    We extend $ \widetilde{A}^{ij} $ to $ \widetilde{M} $ while preserving its ellipticity and $ C^{2,\alpha}  $ regularity. We invoke Proposition \ref{Prop-Green} to construct a series of singular solutions to (\ref{eq:prop5}) parametrized by $ \tau $:
    \begin{equation}\label{eq:sol-green}
        w_\tau(x)=H_n(x-x_\tau)+\omega_{n,\tau} (x-x_\tau)+\delta_{n,\tau}(x-x_\tau) ,
    \end{equation}
    where the remainder terms $ \omega_{n,\tau},\delta_{n,\tau}  $ are given by
    \begin{equation*}
        \omega_{n,\tau} (x):=\omega_n\left(\frac{x}{\sqrt{\det \widetilde{A}(x_\tau)}}\right),\quad \delta_{n,\tau}(x):=H_n\left(\frac{x}{\sqrt{\det \widetilde{A}(x_\tau)}}\right)-H_n(x).
    \end{equation*}
    It is easy to check that $ \omega_{n,\tau} $ satisfies order estimates (\ref{singular_solution-estimate}). Since $ A(x) $ is differentiable and $ \widetilde{A}^{ij} (x_0)=\delta^{ij}  $, for $ \tau<\rho $ we know that
    $$ \delta_{n,\tau}(x)=K(\tau)\tau|x|^{2-n}, $$
    where $ |K(\tau)|\leq K_\rho $ for some $ K_\rho>0 $.

    Let $ w^{h_1}=w^{h_2}=w_\tau $ for $ \tau<r/2$. Then (\ref{eq:2nd-order-8}) becomes
    \begin{equation}\label{eq:2nd-order-8.5}
        \int_{M-B_{\rho}(x_\tau)}\phi \langle\nabla w_\tau,\nabla w_\tau \rangle +\int_{B_{\rho}^+(x_\tau)\cap M}\phi\partial_i w_\tau {A}^{ij} \partial_jw_\tau =0.
    \end{equation}
   Exploiting the structure of (\ref{eq:sol-green}) enables us to uniformly bound the first term of (\ref{eq:2nd-order-8.5}):
    \begin{align*}
        \left|\int_{M-B_{\rho}(x_\tau)}\phi  \langle\nabla w_\tau,\nabla w_\tau \rangle\right|&\leq C\int_{M-B_{\rho/2}(x_0)}|\phi|\left(|\nabla H_n|+|\nabla \omega_{n,\tau}|+|\nabla \delta_{n,\tau}|\right)^2\\
        &\leq (n-2)^2 C\int_{M-B_{\rho/2}(x_0)}|\phi|(1+\widetilde{C}_{\beta,\widetilde{M},(\widetilde{A}^{ij} )}|x|^{\beta}+K_\rho\tau)^2|x|^{2-2n}\\
        &=\overline{C}(\beta,M,(\widetilde{A}^{ij}),\rho).
    \end{align*}
    For the second term of (\ref{eq:2nd-order-8.5}), we make integration by parts. Let $ \widetilde{\phi}:=\phi/a(t,x,0) $, we have
    \begin{align*}
        \lefteqn{\int_{B_{\rho}^+(x_\tau)\cap M}{\phi}\partial_i w_\tau {A}^{ij} \partial_jw_\tau=\int_{B_{\rho}^+(x_\tau)\cap M}\widetilde{\phi}\partial_i w_\tau \widetilde{A}^{ij} \partial_jw_\tau} \\
        &= \int_{B_{\rho}^+(x_\tau)\cap M} \widetilde{\phi}(0)\widetilde{A}^{ij} \partial_i w_\tau\partial_j w_\tau+ \int_{B_{\rho}^+(x_\tau)\cap M} (\widetilde{\phi}-\widetilde{\phi}(0))\widetilde{A}^{ij}\partial_i w_\tau\partial_j w_\tau\\
        &=\widetilde{\phi}(0)\left(\int_{ B_{\rho}(x_\tau)\cap\partial M}+\int_{\partial B_{\rho}^+(x_\tau)\cap M}\right)w_\tau \nu_i \widetilde{A}^{ij} \partial_j w_\tau+\int_{B_{\rho}^+(x_\tau)\cap M} \left(\widetilde{\phi} \widetilde{A}^{ij} -\widetilde{\phi}(0)\widetilde{A}^{ij}\right)\partial_i w_\tau\partial_j w_\tau.
    \end{align*}
    The last identity used Stokes' theorem and $ \partial_i(\widetilde{A}^{ij}\partial_j w_\tau)=0 $ in the interior. Then this term can be split as
    \begin{align*}
        \lefteqn{\int_{B_{\rho}^+(x_\tau)\cap M}{\phi}\partial_i w_\tau {A}^{ij} \partial_jw_\tau} \\
        &=\int_{ B_{\rho}(x_\tau)\cap\partial M}\widetilde{\phi}(0)w_\tau \partial_\nu  w_\tau+\int_{\partial B_{\rho}^+(x_\tau)\cap M}\widetilde{\phi}(0)w_\tau \partial_\nu w_\tau \\
        &\quad +\int_{ B_{\rho}(x_\tau)\cap\partial M}\widetilde{\phi}(0)w_\tau \nu_i( \widetilde{A}^{ij}-\delta^{ij})\partial_j  w_\tau+\int_{\partial B_{\rho}^+(x_\tau)\cap M}\widetilde{\phi}(0)w_\tau \nu_i (\widetilde{A}^{ij}-\delta^{ij})\partial_j w_\tau \\
        &\quad +\int_{B_{\rho}^+(x_\tau)\cap M} \left(\widetilde{\phi} \widetilde{A}^{ij} -\widetilde{\phi}(0)\widetilde{A}^{ij}\right)\partial_i w_\tau\partial_j w_\tau\\
        &=I_1(\tau)+I_2(\tau)+I_3(\tau)+I_4(\tau)+I_5(\tau).
    \end{align*}\par
    To establish the order estimates for $ I_l(\tau),~l=1,\dots,5 $ as $ \tau \to 0 $, we first show that for $ k<1-n $ and $ 0<\tau<\rho/2 $,
    \begin{equation}\label{int_estimate}
        C_1\tau^{n-1+k} \leq \int_{B_{\rho}^{n-1}(x_\tau)}|x-x_\tau|^k \leq C'_1 \tau^{n-1+k}.
    \end{equation}
    This is by computing
    \begin{equation*}
        \begin{aligned}
            \int_{B_{\rho}^{n-1}(x_\tau)}|x-x_\tau|^k&=\int_{|(x',0)-(0,-\tau)|<\rho}|(x',0)-(0,-\tau)|^k dx' \\
            &= |\mathbb{S}^{n-2}| \int_0^{(\rho^2-\tau^2)^{1/2}} (r^2+\tau^2)^{\frac{k}{2}}r^{n-2} dr\\
            &= \tau^{n-1+k} |\mathbb{S}^{n-2}| \int_0^{(\rho^2/\tau^2-1)^{1/2}} (s^2+1)^{\frac{k}{2}}s^{n-2} ds.
        \end{aligned}
    \end{equation*}
    Actually, we take
    \begin{equation*}
        C_1:=|\mathbb{S}^{n-2}| \int_0^{\sqrt{3}} (s^2+1)^{\frac{k}{2}}s^{n-2} ds,\quad C'_1:=|\mathbb{S}^{n-2}| \int_0^{\infty} (s^2+1)^{\frac{k}{2}}s^{n-2} ds.
    \end{equation*}

    We next prove that term $ I_1 $ diverges as $ \tau\to 0 $ if $ \widetilde{\phi}(0)\neq 0 $. To see this, we plug (\ref{eq:sol-green}) into $ |I_1| $. For $ x\in B_{\rho}(x_\tau)\cap\partial M $, we have $ \nu\cdot (x-x_\tau)=-\tau $, and thus according to (\ref{singular_solution-estimate}) and (\ref{int_estimate}),
    \begin{align*}
        \partial_\nu w_\tau(x)&=\partial_\nu\left(|x-x_\tau|^{2-n} +\delta_{n,\tau} (x-x_\tau)+\omega_{n,\tau} (x-x_\tau)\right)\\
        &=(n-2)\tau|x-x_\tau|^{-n}+(n-2)K(\tau)\tau^{2}|x-x_\tau|^{-n} +O(|x-x_\tau|^{1-n+\beta}),
    \end{align*}
    where $ 0<\beta<1 $. As a result, we get the lower bound of $ |I_1| $ by
    \begin{align*}
        |I_1|&\geq(n-2)|\widetilde{\phi}(0)|\cdot\int_{B_{\rho}^{n-1}(x_\tau)}\left((\tau+K(\tau)\tau^2)|x-x_\tau|^{2-2n}-C|x-x_\tau|^{3-2n+\beta}\right) dx,\\
        &\geq(n-2)|\widetilde{\phi}(0)|(C_1(1-K_{\rho}\tau )\tau^{2-n}-{C}_2\tau^{2-n+\beta})\\
        &\geq C_3|\widetilde{\phi}(0)|\tau^{2-n} ,\quad\text{for}~\tau <<1.
    \end{align*}\par
    For other terms, we show that
    \begin{equation*}
        \lim_{\tau\to 0}\tau^{n-2} |I_l(\tau)|=0,\quad l=2,\dots,5.
    \end{equation*}
    Since $ |x-x_\tau|=\rho $ on $ \partial B_{\rho}^+(x_\tau)\cap M $, the term $ I_2(\tau) $ can be controlled by some constant only dependent on $ \rho $ using the estimate (\ref{singular_solution-estimate}):
    \begin{equation*}
        |I_2|\leq(n-2)\text{Area}(\partial B_{\rho})\cdot \left(1+K_{\rho}\rho +C_{\beta,\overline{\Omega},(\widetilde{A}^{ij} )}\rho^{\beta}\right)^2 \rho^{3-2n} \leq \tilde{C}(\rho) .
    \end{equation*}
    Similarly, $ |I_4| $ is also uniformly bounded.\par
    Then we estimate $ |I_3| $. Since $ \widetilde{A}^{ij} $ is differentiable, there exists a constant $ C'(\rho) $ such that
    \begin{equation*}
        |\widetilde{A}^{ij}(x)-\delta^{ij} |\leq C'(\rho)|x|\leq C'(\rho)(|x-x_\tau|+\tau).
    \end{equation*}
    Plugging (\ref{singular_solution}) into $ |I_3| $, we obtain
    \begin{align*}
        |I_3|&\leq C'(\rho)|\widetilde{\phi}(0)|\cdot\int_{ B_{\rho}(x_\tau)\cap\partial M}  (|x-x_\tau|+\tau)|w_\tau||\nabla w_\tau|dx\\
        &=C'(\rho)|\widetilde{\phi}(0)|\cdot\int_{ B_{\rho}(x_\tau)\cap\partial M}  (|x-x_\tau|+\tau)\left|(1+K(\tau)\tau)|x-x_\tau|^{2-n} +\omega_{n,\tau} (x-x_\tau)\right|\\
        &\qquad \qquad \qquad \qquad \qquad \qquad \cdot\left|\nabla\left((1+K(\tau)\tau)|x-x_\tau|^{2-n}+\omega_{n,\tau} (x-x_\tau)\right)\right|dx.
    \end{align*}
    Using (\ref{singular_solution-estimate}) and (\ref{int_estimate}) gives
    \begin{align*}
        |I_3|&\leq (n-2)C'(\rho)|\widetilde{\phi}(0)|\cdot\int_{B_{\rho}^{n-1}(x_\tau)}\left(\tau(1+K_{\rho} \tau)^2|x-x_\tau|^{3-2n}+(1+K_{\rho} \tau)^2|x-x_\tau|^{4-2n}\right.\\
        &\qquad \qquad \qquad \qquad \qquad \qquad \qquad \left.+C'\tau|x-x_\tau|^{3-2n+\beta}+C''|x-x_\tau|^{4-2n+\beta}\right)dx\\
        &\leq 2(n-2)C''(\rho)|\widetilde{\phi}(0)|\tau^{3-n}+O(\tau^{3-n+\beta} ).
    \end{align*}

    Noticing that there exists a constant $ C(\rho) $ such that $$ |\widetilde{\phi}(x)\widetilde{A}^{ij}(x)-\widetilde{\phi}(0)\widetilde{A}^{ij}(x)|\leq C(\rho)|x| $$ holds for $ x\in B_{\rho}(0) $. Applying (\ref{singular_solution-estimate}) and (\ref{int_estimate}) again, we obtain the upper bound estimate of $ |I_5(\tau)| $:
    \begin{equation*}
        \begin{aligned}
            |I_5|&\leq C(\rho)\int_{B_{\rho}^+ (x_\tau)\cap M} |x||\nabla w_\tau|^2\\
            &\leq C(\rho)\int_{B_{\rho+\tau}^+ (0)}(|x-x_\tau|+\tau)\cdot\left|\nabla\left((1+K(\tau)\tau)|x-x_\tau|^{2-n}+\omega_{n,\tau} (x-x_\tau)\right)\right|^2\\
            &\leq (n-2)C(\rho)\int_{B_{\rho+\tau}^+ (0)}\left(\tau(1+K(\tau)\tau)^2|x-x_\tau|^{2-2n}+(1+K(\tau)\tau)|x-x_\tau|^{3-2n}\right.\\
            &\qquad \qquad\left.+C'\tau|x-x_\tau|^{2-2n+\beta}+C''|x-x_\tau|^{3-2n+\beta}\right)\\
            &\leq C_4 \tau^{3-n} .
        \end{aligned}
    \end{equation*}

    When we let $ \tau $ decrease to zero, the order estimates above show that $$ \phi(x_0)=a(t,0,0)\widetilde{\phi}(0)=0. $$
    Since $ x_0 $ is arbitrarily chosen, the function $ \phi $ must vanish on the whole boundary.
\end{proof}



Next, we extend the result to the interior.

\begin{proof}[Proof of Proposition \ref{Prop-basic} in the interior]
We have proved that
\begin{equation}\label{eq:2nd-order-7}
    \phi(x)=0,\quad\forall x\in\partial M.
\end{equation}
Applying Stokes' theorem to (\ref{eq:2nd-order-8}) with (\ref{eq:2nd-order-7}) gives
\begin{align*}
    0&=\int_M \phi A^{ij} \partial_j w^{h_1} \partial_i w^{h_2} dx=\int_M \widetilde{\phi} \widetilde{A}^{ij} \partial_j w^{h_1} \partial_i w^{h_2} dx\\
    &=-\int_M \partial_i\widetilde{\phi} \widetilde{A}^{ij} \partial_j w^{h_1} w^{h_2} dx-\int_M \widetilde{\phi} \partial_i(\widetilde{A}^{ij} \partial_j w^{h_1})w^{h_2} dx.
\end{align*}
Since $ w^{h_1}  $ is a harmonic function on $ (M,g) $, we have that for any $ h_1,h_2\in C^{\infty} (\partial M) $,
\begin{equation*}
    \int_M \partial_i\widetilde{\phi} \widetilde{A}^{ij} \partial_j w^{h_1} w^{h_2} dx=0.
\end{equation*}
Therefore, by switching $ h_1 $ and $ h_2 $ we know that
\begin{equation}\label{eq:2nd-order-4}
    \int_M \partial_i\widetilde{\phi} \widetilde{A}^{ij} \partial_j (w^{h_1} w^{h_2})dx=\int_M \partial_i\widetilde{\phi} \widetilde{A}^{ij} \partial_j w^{h_1} w^{h_2} dx+\int_M \partial_i\widetilde{\phi} \widetilde{A}^{ij} \partial_j w^{h_2} w^{h_1}  dx=0.
\end{equation}

Then we apply to (\ref{eq:2nd-order-4}) the following lemma from \cite[Lemma 4.7]{sun1997inverse}.
\begin{lemma}\label{lmm:tangent}
    Let $ (M,\widetilde{g}) $ be a compact manifold with boundary and $ X^j\in C^1(M)(j=1,2,\cdots,n) $ be a vector field. If
    \begin{displaymath}
        \int_M X^j\partial_j (w^{h_1} w^{h_2})dx=0
    \end{displaymath}
    holds for arbitrary $ w^{h_1} , w^{h_2}\in \mathcal{H}_{\widetilde{g}}(M)  $, then $ X(x) \in T_x(\partial M) $ for all $ x \in \partial M $.
\end{lemma}
By taking $ \widetilde{g}=a(t,x,0)^{\frac{2}{n-2}}(e\oplus g_0) $, we have $ \sqrt{\det \widetilde{g}}(\widetilde{g}_{ij})^{-1} =\widetilde{A}^{ij}  $, and hence $ \mathcal{H}_{\widetilde{g}}(M) = \{\partial_i(\widetilde{A}^{ij}\partial_j w )=0\} $. By Lemma \ref{lmm:tangent}, we obtain that $ \nu_j\partial_i\widetilde{\phi} \widetilde{A}^{ij}=0 $ on $ \partial M $. Performing integration by parts on (\ref{eq:2nd-order-4}) shows
\begin{equation}\label{eq:2nd-order-5}
    \int_M \partial_j(\widetilde{A}^{ij}\partial_i\widetilde{\phi}) w^{h_1} w^{h_2}dx=0.
\end{equation}
Under the CTA geometric condition, Lemma \ref{lmm:density-cta} yields
\begin{equation}\label{eq:2nd-order-6}
    \partial_j[\widetilde{A}^{ij} (x)\partial_i\widetilde{\phi}(x)] \equiv 0.
\end{equation}
Equation (\ref{eq:2nd-order-6}) and boundary value (\ref{eq:2nd-order-7}) form a Dirichlet problem of elliptic equations. By the uniqueness of the solution, we know that $ \widetilde{\phi} \equiv 0 $, and hence $ \phi \equiv 0 $ on $ M $.
\end{proof}
\begin{remark}
    As is shown in \cite{sun1997inverse}, the proof of Lemma \ref{lmm:tangent} also results from taking singular solutions in Alessandrini's paper. However, we need a small adaptation to get a suitable order estimate as in the proof of Proposition \ref{Prop-basic}. In fact, one needs to take $ w_\tau(x)=w(x-x_\tau) $ to be
    \begin{equation*}
        w(x)=|x|^{2-n-m}S_m\left(\frac{x}{|x|}\right)+\omega_{n,m}(x),
    \end{equation*}
    with $ m\geq 1 $, $ S_m $ the spherical harmonic function of degree $ m $, and 
    \begin{equation*}
        |\omega_{m,n}(x)|+|x||\nabla\omega_{n,m}(x)|\leq C|x|^{2-n-m+\alpha}.
    \end{equation*}
    We refer the reader to \cite[][Theorem 1.1]{alessandrini1990singular}.
\end{remark}
Now it remains to prove $ T_0^{l} \equiv 0 $ for $1\leq l\leq n$ for details.

\begin{proposition}\label{Prop-1st}
Under the assumptions of Theorem \ref{thm:main},
$$ \partial_{p^l}a_1(0,x,0)=\partial_{p^l}a_2(0,x,0),\quad 1\leq l\leq n.$$ Together with Proposition \ref{Prop-basic}, equation (\ref{deriv-equal}) holds for $ |\beta|=1 $.
\end{proposition}\par
\begin{proof}
Proposition \ref{Prop-basic} has proved $ T_0^0\equiv 0 $. 
Hence, equation (\ref{eq:2nd-order-2}) reduces to
\begin{equation}\label{eq:2nd-order-prop}
    \sum_{\sigma\in\pi(2)}\int_M T_0^l \partial_l w^{h_{\sigma(1)}}\langle \nabla w^{h_{\sigma(2)}}, \nabla w^{h_3}\rangle=0.
\end{equation}
We apply the same procedure in \cite[][(4.6-4.10)]{carstea2021calderon} to simplify this equation. Taking $ h_1=h_2=h $ in (\ref{eq:2nd-order-prop}) gives 
\begin{equation}\label{eq:2nd-order-partial}
    \int_M T_0^l\partial_l w^h \langle \nabla w^{h},  \nabla w^{h_3}\rangle=0,\quad\forall h\in C^{\infty}(\partial M). 
\end{equation}

Letting $ h_1=h_2=h+h_3\in C^{\infty}(\partial M)  $ in (\ref{eq:2nd-order-prop}), we obtain
\begin{equation}\label{eq:2nd-order-symmetric}
    \begin{aligned}
        0&=\int_M T_0^l\partial_l (w^{h} +w^{h_3})\langle\nabla (w^{h} +w^{h_3}), \nabla w^{h_3}\rangle\\
        &=\int_M T_0^l\partial_l w^{h}\langle\nabla w^{h_3},\nabla w^{h_3}\rangle
        +\int_M T_0^l\partial_l w^{h_3}\langle\nabla w^{h_3},\nabla w^{h+h_3}\rangle
        +\int_M T_0^l\partial_l w^{h}\langle\nabla w^{h},\nabla w^{h_3}\rangle\\
        &=\int_M T_0^l\partial_l w^{h}\langle\nabla w^{h_3},\nabla w^{h_3}\rangle,
    \end{aligned}
\end{equation}
where the last identity results from (\ref{eq:2nd-order-partial}). By substituting $ h_4+h_5 $ for $ h_3 $ in (\ref{eq:2nd-order-symmetric}), we obtain that
\begin{equation}\label{eq:2nd-order-3}
    \begin{aligned}
        &\int_M T_0^l\partial_l w^{h}\langle\nabla w^{h_5},\nabla w^{h_5}\rangle+\int_M T_0^l\partial_l w^{h}\langle\nabla w^{h_4},\nabla w^{h_4}\rangle+2\int_M T_0^l\partial_l w^{h}\langle\nabla w^{h_4},\nabla w^{h_5}\rangle\\
        &=2\int_M T_0^l\partial_l w^{h}\langle\nabla w^{h_4},\nabla w^{h_5}\rangle =0
    \end{aligned}
\end{equation}
by applying (\ref{eq:2nd-order-symmetric}) for $ h_3=h_4 $ and $ h_3=h_5 $. Furthermore, Proposition \ref{Prop-basic} gives
\begin{equation*}
    T_0^l(x)\partial_l w^{h}(x)\equiv 0,\quad\forall  h\in C^{\infty} (M).
\end{equation*}\par
According to \cite[][Proposition A.5]{lassas2020poisson}, for an arbitrary tangent vector $ \xi \in T_{x_0}M $ there exists a Dirichlet boundary value $ h $ such that $ \nabla w^h(x_0)=\xi $. Thus, by arbitrariness of the choice of $ x_0 $ it holds that $ T_0^l(x)\equiv 0 $, namely
\begin{displaymath}
    \partial_{p^l}a_1(0,x,0)=\partial_{p^l}a_2(0,x,0)\quad(1\leq l \leq n).
\end{displaymath}
This completes the proof.
\end{proof}



Since $ T_0^\alpha\equiv 0 $, $ \alpha=0,\cdots,n $, the right-hand side of the equation (\ref{eq:2nd-order-1}) vanishes. By uniqueness of elliptic equations, it must hold that $ w_{0}^{h_1,h_2}\equiv 0 $, which means that
\begin{equation}\label{eq:2nd-order-sol}
    w^{h_1,h_2} := w_{1}^{h_1,h_2}=w_{2}^{h_1,h_2}.
\end{equation}

\section{Recover the Higher-Order Derivatives}
In this section, our goal is to prove that we can recover all the higher-order derivatives of the factor $ a_m(s,x,p) $ with the knowledge of $ N_{\gamma_m} $ under CTA conditions. The key observation here is that the overall integral of products of gradients of solutions can be simplified if we construct certain Complex Geometric Optics (CGO) solutions and calculate the asymptotic behavior of the integrals.\par

\subsection{Third-Order Linearization: Constructing CGO Solutions}
Our task in this subsection is to prove the following proposition for recovery of the second-order derivatives of the quasilinear factor $ a(s,x,p) $. 
\begin{proposition}\label{Prop-2nd}
Under the assumptions of Theorem \ref{thm:main}, equation (\ref{deriv-equal}) holds for $ |\beta|=2 $.
\end{proposition}

For the order of linearization $ L=3 $, we differentiate (\ref{eq:quasi_intrinsic}) by $ \partial_{\epsilon^1}\partial_{\epsilon^2}\partial_{\epsilon^3}  $ at $ \epsilon=0 $, and obtain the equation (\ref{eq:higher-order-0}) for $ a_1 $ and $ a_2 $. Let $ h_1,h_2,h_3\in C^{\infty} (\partial M) $ and $ w_{m}^{h_1,h_2,h_3} $ be the solution to (\ref{eq:higher-order-0}) for $a(s,x,p)=a_m(s,x,p),m=1,2$. 
Consider the function $$ w_{0}^{h_1,h_2,h_3}:=w_{1}^{h_1,h_2,h_3}-w_{2}^{h_1,h_2,h_3} $$
and the tensor as defined in (\ref{def:T})
$$ T_0^{\alpha_1,\alpha_2}(x):=T_{1}^{\alpha_1,\alpha_2}(x,0)-T_{2}^{\alpha_1,\alpha_2}(x,0). $$ 
As (\ref{eq:2nd-order-sol}) holds, the difference between (\ref{eq:higher-order-0}) for $ a_1,a_2 $ is thus given by
\begin{equation}\label{eq:3rd-order-1}
    \begin{cases}
        \nabla^* \left(a(t,x,0) \nabla w_{0}^{h_1,h_2,h_3}\right)=-\sum\limits_{\sigma\in\pi(3)}\nabla^*\left(T_0^{\alpha_1,\alpha_2}P_{\alpha_1} w^{h_{\sigma(1)}}P_{\alpha_2}w^{h_{\sigma(2)}}\nabla w^{h_{\sigma(3)}}\right)&\text{\rm in}~M,\\
        w_{0}^{h_1,h_2,h_3}=0&\text{\rm on}~\partial M,
    \end{cases}
\end{equation}
and the difference of the corresponding DN maps (\ref{eq:higher-order-0-DN}) is
\begin{equation}\label{eq:3rd-order-1.5}
    0=\nu\cdot \big(a(t,x,0)\nabla w_{0}^{h_1,h_2,h_3}+\sum\limits_{\sigma\in\pi(3)}T_0^{\alpha_1,\alpha_2}P_{\alpha_1} w^{h_{\sigma(1)}}P_{\alpha_2}w^{h_{\sigma(2)}}\nabla w^{h_{\sigma(3)}}\big)\quad \text{\rm on}~\partial M.
\end{equation}\par

We apply Stokes' theorem to the weak form of (\ref{eq:3rd-order-1.5}) with $ h_4\in C^\infty(\partial M) $ as the computation from (\ref{eq:2nd-order-IBP}) to (\ref{eq:2nd-order-2}). This gives that for any $ h_1,h_2,h_3,h_4\in C^\infty(\partial M) $,
\begin{equation}\label{eq:3rd-order-2}
    \sum_{\sigma\in\pi(3)}\int_M T_0^{\alpha_1,\alpha_2}P_{\alpha_1} w^{h_{\sigma(1)}}P_{\alpha_2} w^{h_{\sigma(2)}}\langle \nabla w^{h_{\sigma(3)}}, \nabla w^{h_4} \rangle=0.
\end{equation}
By the smoothness of $ a_m$, we know that $ T_0^{\alpha_1,\alpha_2} $ is symmetric.\par 

For $ \alpha_1=\alpha_2=0 $, we can take $ h_1=h_2=1 $, and whence $ w^{h_1}=w^{h_2} =1 $. Consequently, (\ref{eq:3rd-order-2}) becomes
\begin{equation*}
    2\int_{M} T_0^{0,0} \langle \nabla w^{h_{3}}, \nabla w^{h_4} \rangle=0,\quad\forall h_3,h_4. 
\end{equation*}
Leveraging Proposition \ref{Prop-basic}, we obtain $ T_0^{0,0}\equiv 0 $.\par
For $ \alpha_1=0 $, we can take $ h_3=1 $, whence $ w^{h_3}=1 $. As a result, (\ref{eq:3rd-order-2}) becomes
\begin{equation}\label{eq:3rd-order-2-1}
    2\sum_{\sigma\in\pi(2)}\int_M T_0^{0,l_2}\partial_{l_2} w^{h_{\sigma(1)}}\langle \nabla w^{h_{\sigma(2)}},\nabla w^{h_4} \rangle =0.
\end{equation}
Actually, (\ref{eq:3rd-order-2-1}) is the analogue of (\ref{eq:2nd-order-prop}) in the proof of Proposition \ref{Prop-1st}. Following the same argument therein, we can prove that $ T_0 ^{0,l_2}\equiv 0 $ for all $ 1\leq l_2\leq n $.\par 

Finally, it remains to handle the terms with $ T_{0} ^{j,k} (1\leq j,k \leq n)$ in (\ref{eq:3rd-order-2}). Let $ T $ be a symmetric (0,2)-tensor given by
$$ T(X,Y):=T_{0}^{j,k}g_{jl} g_{km} X^l Y^m,~\mbox{for any $ X,Y\in \mathbb{R}^n $}.$$\par
For any $ y_0\in M_0^{int}  $, let $ \Gamma^{(1)} $ and $ \Gamma^{(2)}  $ be two distinct non-tangential geodesics intersecting at $ y_0\in M_0^{int}  $. By simplicity of $ M_0 $, they intersect only once and do not self-intersect. Proposition \ref{Prop-GBQ} shows that there exist Gaussian beam quasi-modes $ e^{is\psi^{(1)}(x')}b^{(1)}_{s}(x')$ and $ e^{is\psi^{(2)}(x')}b^{(2)}_{s}(x') $ concentrated near $ \Gamma^{(1)} $ and $ \Gamma^{(2)}  $ respectively. Subsequently, Proposition \ref{Prop-CGO} shows the existence of the following two pairs of CGO solutions to (\ref{eq:1st-order-int}) in $ (M,g) $, :
\begin{equation}\label{eq:cgo-choice}
    \left\{\begin{aligned}
        w^{1}(x)& :=u_{\tau+i\lambda}^{(1)}(x) = e^{-(\tau+i\lambda)x^1} \left(e^{i(\tau+i\lambda)\psi^{(1)}(x')}b^{(1)}_{\tau+i\lambda}(x')+r^{(1)} _{\tau+i\lambda}\right), \\
        w^{2}(x)& :=\overline{u_{\tau+i\lambda}^{(1)}(x)},\\
        w^{3}(x)& :=u_{-\tau}^{(2)}(x) = e^{\tau x^1} \left(e^{-i\tau\psi^{(2)}(x')}b^{(2)}_{-\tau}(x')+r^{(2)} _{-\tau}\right),\\
        w^{4}(x)& :=\overline{u_{-\tau}^{(2)}(x)},
    \end{aligned}\right.
\end{equation}
for any $ x=(x^1,x')\in (\mathbb{R}\times M_0,c \cdot (e\oplus g_0))$. 
Now we take $ b_{\tau+i\lambda}^{(1)},b_{-\tau}^{(2)} $ in the form of (\ref{eq:gbq-wkb}), so that $$ \| r^{(i)}_s\|_{C^1(\overline{M})}\leq \| r^{(i)}_s\|_{H^k(M)}=O(s^{-K}),~\mbox{where }K\geq 1,~k> 1+n/2.$$

Let $ \partial_1 = (1,0,\cdots,0)\in\mathbb{R}^n $. Taking $ w^{h_j}=w^j  $, $ j=1,2,3,4 $ in (\ref{eq:3rd-order-2}) and applying Proposition \ref{Prop-GBQ} and \ref{Prop-CGO}, we get that
\begin{equation}\label{eq:3rd-order-approx}
    \begin{aligned}
        0&= \tau^{n+2} \int_{M_0} \left|b_{(0)}^{(1)}\right|^2\left|b_{(0)}^{(2)}\right|^2 \int_{\mathbb{R}} e^{2\text{\rm Im}\lambda x^1-2\tau\text{\rm Im}(\psi^{(1)} (x')+\psi^{(2)} (x'))}l(x^1,x')dx^1 dx'+O(\tau^{\frac{n+4}{2}} )
    \end{aligned}
\end{equation}
as $ |\tau|\to\infty $, where
\begin{equation}\label{eq:3rd-split}
    \begin{split}
        l(x^1,x'):=e^{-2\text{\rm Re}\lambda\text{\rm Re}\psi^{(1)} (x')}&\cdot\bigg[T(\partial_1+i\nabla\psi^{(1)},\partial_1-i\nabla\psi^{(1)}) \langle\partial_1+i\nabla\psi^{(2)},\partial_1-i\nabla\psi^{(2)}\rangle\\
        & +T(\partial_1+i\nabla\psi^{(1)},\partial_1+i\nabla\psi^{(2)})\langle \partial_1-i\nabla\psi^{(1)},\partial_1-i\nabla\psi^{(2)} \rangle\\
        & +T(\partial_1-i\nabla\psi^{(1)},\partial_1+i\nabla\psi^{(2)}) \langle\partial_1+i\nabla\psi^{(1)},\partial_1-i\nabla\psi^{(2)}\rangle \bigg].
    \end{split}
\end{equation}
We show that $ l $ vanishes at any $ (x^1,y_0)\in M $ by proving the following lemma:
\begin{lemma}\label{lmm:4-Gaussian-beams}
    Let $ (M,g)\subset(\mathbb{R}\times M_0,c \cdot (e\oplus g_0)) $ be a compact simple CTA manifold and $ l\in C(M) $ be as in (\ref{eq:3rd-split}). Let two non-tangential geodesics $ \Gamma^{(i)}~(i=1,2) $ intersect at $ y_0\in M_0 $, and two correlating Gaussian beam quasi-modes $ e^{is\psi^{(i)}(x') } b_s^{(i)}(x')~(i=1,2)  $ be as constructed in Proposition \ref{Prop-GBQ}. If $ l $ satisfies (\ref{eq:3rd-order-approx}) for any $ \lambda\in\{\lambda\in\mathbb{C}:~|\lambda|<1\},~\delta>0 $, with $ b^{(1)}_{(0)}$ and $ b^{(2)}_{(0)} $ supported within $ \delta $-tubular neighborhoods of the two geodesics respectively, then $ l= 0 $ at $ (t,y_0) $ for any $ (t,y_0)\in M $.
\end{lemma}

\begin{proof}
We first prove that for any $ \lambda\in\mathbb{C},~|\lambda|<1,~y_0\in M_0 $ the function
\begin{equation}\label{eq:3rd-order-h}
    h_{\text{\rm Im}\lambda} (x'):=\int_{\mathbb{R}} e^{2\text{\rm Im}\lambda x^1}l(x^1,x')dx^1,
\end{equation}
vanishes at $ y_0 $. We prove by contradiction. Suppose $ z_0:=\mathrm{Re}\,h_{\mathrm{Im}\lambda} (y_0)> 0 $. By continuity of $ h_{\mathrm{Im}\lambda} $ there exists a neighborhood $ N_{2\delta} (y_0)\subseteq M_0 $ such that $ \mathrm{Re}\,h_{\mathrm{Im}\lambda}(N_{2\delta} (y_0))\subseteq(z_0/2,3z_0/2)$. By Proposition \ref{Prop-CGO} and simplicity of $ M_0 $, there exist small tubular neighborhoods $ N_{\delta}(\Gamma^{(i)}) $ of $ \Gamma^{(i)}  $ such that $ \mathrm{supp}\, b_{(0)} ^{(i)}\subseteq N_{\delta}(\Gamma^{(i)}) $ and that $ N_{\delta}(\Gamma^{(1)})\cap N_{\delta}(\Gamma^{(2)})\subseteq N_{2\delta} (y_0) $. Therefore, for $ \tau $ large enough, we have
\begin{equation*}
    \begin{aligned}
        &\mathrm{Re}\int_{M_0} e^{-2\tau\text{\rm Im}(\psi^{(1)} (x')+\psi^{(2)} (x'))}\left|b_{(0)}^{(1)}\right|^2\left|b_{(0)}^{(2)}\right|^2 h_{\text{\rm Im}\lambda} (x')dx'\\
        &=\int_{N_{\delta}(\Gamma^{(1)})\cap N_{\delta}(\Gamma^{(2)})} e^{-2\tau\text{\rm Im}(\psi^{(1)} (x')+\psi^{(2)} (x'))}\left|b_{(0)}^{(1)}\right|^2\left|b_{(0)}^{(2)}\right|^2 \mathrm{Re}\,h_{\text{\rm Im}\lambda} (x')dx'\geq C_0\tau^{-\frac{n-2}{2}} > 0.
    \end{aligned}
\end{equation*}
The inequality comes from the fact that $ \nabla^2\text{\rm Im}(\psi^{(i)})=\text{\rm Im}\nabla^2(\psi^{(1)})\geq 0 $ as shown in Proposition \ref{Prop-GBQ}. Consequently, the leading term in (\ref{eq:3rd-order-approx}) is of order $ \tau^{n+2-\frac{n-2}{2}}=\tau^{\frac{n+6}{2}}   $, which leads to a contradiction when $ \tau $ goes to infinity. Through similar processes, we can show that both $ \mathrm{Re}\,h_{\mathrm{Im}\lambda} (y_0) $ and $ \mathrm{Im}\,h_{\mathrm{Im}\lambda} (y_0) $ must be zero.

Secondly, we prove that $ l(x^1,y_0) $ vanishes for any $ x^1\in \mathbb{R} $. Since the manifold $ M\subset\subset \mathbb{R}\times M_0 $ is compact, for any $ x'\in M_0 $ the function $ l(x^1,x') $ of $ x^1 $ is supported on some finite interval $ I_{x'}\subset\subset\mathbb{R}  $. Notice that $ h_L(y_0)=0 $ for all $ L=\text{\rm Im}\lambda\in\mathbb{R},|\lambda|\leq 1$, namely
\begin{equation*}
    F(L):=\int_{I_{y_0} } l(x^1,y_0)e^{2Lx^1} dx^1=0,~ \forall -1<L<1.
\end{equation*}
Since $ F(L) $ is smooth in $ L $, we have $ F^{(k)} (0)=0 $ for all $ k\in\mathbb{N} $, which means
\begin{equation*}
    \int_{I_{y_0} } (2x^1)^k l(x^1,y_0) dx^1=0.
\end{equation*} 
By the Stone-Weierstrass theorem, the functions $ \{(x^1)^k\}_{k\in\mathbb{N}} $ span a dense subspace in $ C(I_{y_0}) $. Hence, for any function $ h\in C(I_{y_0} ) $,
\begin{equation*}
    \int_{I_{y_0}} l(x^1,y_0) h(x^1)dx^1=0.
\end{equation*}
Therefore, $ l(x^1,y_0) $ must be zero for all $ x^1 \in I_{y_0} $.
\end{proof}

In the following, we view $ T_{y_0}M_0  $ as a natural subspace of $ T_{(x^1,y_0)}M  $. For any $ y_0\in M_0 $, let $V_{y_0} \subseteq T_{y_0} M_0$ denote the collection of unit tangent vectors of non-tangential geodesics at $ y_0 $ by 
\begin{equation}\label{def:geodim}
    V_{y_0} := \{\dot{\Gamma}\in T_{y_0} M_0:\mbox{$ \Gamma $ is a non-tangential geodesic in $ M_0 $ passing $ y_0 $},~|\dot{\Gamma}|=1\}.
\end{equation}
Since $ M_0 $ is simple, all the geodesics passing $ y_0 $ are non-tangential as in Definition \ref{def:non-tan-geodesic}. Therefore, $ V_{y_0}$ contains all the unit vectors in $ T_{y_0}M_0 $, which implies that $ V_{y_0}=S_{y_0}M= \mathbb{S}^{n-2} $.


Now let $ \xi^{(m)}\in V_{y_0} $ denote the tangent vector $  \nabla\psi^{(m)}(y_0)=\dot{\Gamma}^{(m)}|_{\Gamma^{(m)}=y_0 } $ for $ m=1,2 $. Since we know $ l\equiv 0 $ by Lemma \ref{lmm:4-Gaussian-beams}, (\ref{eq:3rd-split}) gives that at any $(x^1,y_0)\in M$,
\begin{multline}\label{eq:3rd-order-3}
    4T_0^{1,1}+2T(\xi^{(1)},\xi^{(1)})+2T(\xi^{(1)},\xi^{(2)})\langle \xi^{(1)},\xi^{(2)} \rangle\\
    -i2T(\partial_1,\xi^{(1)})\langle \xi^{(1)},\xi^{(2)} \rangle+i2T(\partial_1,\xi^{(2)})=0.
\end{multline}


Therefore, the imaginary part of (\ref{eq:3rd-order-3}) vanishes. By symmetry of $ \xi^{(1)}  $ and $ \xi^{(2)}  $ we know that
\begin{equation*}
    T(\partial_1,\xi^{(2)}-\langle \xi^{(1)},\xi^{(2)} \rangle\xi^{(1)})=0=T(\partial_1,\xi^{(1)}-\langle \xi^{(2)},\xi^{(1)} \rangle\xi^{(2)}).
\end{equation*}
A simple calculation gives that
\begin{equation}
    (1-\langle \xi^{(1)},\xi^{(2)} \rangle^2)T(\partial_1,\xi^{(1)})=0,\quad\forall \xi^{(1)},\xi^{(2)}\in V_{y_0},\xi^{(1)}\neq\pm\xi^{(2)}.
\end{equation}
Since $ V_{y_0}=\mathbb{S}^{n-2} $, it holds that $ |\langle \xi^{(1)},\xi^{(2)} \rangle|<1 $. As a result, 
\begin{equation}\label{eq:3rd-order-4}
    T(\partial_1,\xi)=0,\quad\forall\xi\in T_{y_0} M_0 .
\end{equation}
In other words, $ T_0^{1,k}=T_0^{k,1}=0  $ for $2\leq k\leq n $.

The vanishing of the real part of  (\ref{eq:3rd-order-3}) gives
\begin{equation}\label{eq:3rd-order-5}
    2T_0^{1,1}+T(\xi^{(1)},\xi^{(1)})=-\langle \xi^{(1)},\xi^{(2)} \rangle T(\xi^{(1)},\xi^{(2)}),
\end{equation}
which implies by symmetry that $ T(\xi^{(1)},\xi^{(1)} )=T(\xi^{(2)},\xi^{(2)}) $ for any $ \xi^{(1)},\xi^{(2)}\in V_{y_0}  $. Since $  V_{y_0}=\mathbb{S}^{n-2} $, there exists some $ K=K(x^1,y_0)\in\mathbb{R} $ such that
\begin{equation}\label{eq:3rd-order-6}
    T(\xi,\eta)=K\langle \xi,\eta \rangle,\quad\forall\xi,\eta\in T_{y_0} M_0.
\end{equation}
Plugging this into (\ref{eq:3rd-order-5}), we know that
\begin{equation}\label{eq:3rd-order-7}
    2T_0^{1,1}+(1+\langle \xi,\eta\rangle^2)K=0,\quad\forall\xi,\eta\in V_{y_0},\xi\neq\pm\eta.
\end{equation}
Since $ V_{y_0}=\mathbb{S}^{n-2} $, $ K $ and $ T_0^{1,1} $ in (\ref{eq:3rd-order-7}) must be zero. Combining this with (\ref{eq:3rd-order-6}), we know that $ T_0^{j,k} =0 $ for $ 2\leq j,k\leq n $. 

To conclude, we have proved $ T_{0}^{\alpha_1,\alpha_2}=0 $, which means
\begin{equation}
    Q_{\alpha_1} Q_{\alpha_2}  a_1(t,x,0)=Q_{\alpha_1} Q_{\alpha_2}  a_2(t,x,0).
\end{equation}
If we insert this into (\ref{eq:3rd-order-1}), we will get that $ w_{0}^{h_1,h_2,h_3}=0 $. In other words, we can define
\begin{equation}
    w^{h_1,h_2,h_3}:= w_{1}^{h_1,h_2,h_3}=w_{2}^{h_1,h_2,h_3}.
\end{equation}
\subsection{Higher-Order Linearization: Proof by Induction on Orders}
In this subsection, we prove the following proposition for recovery of higher-order derivatives.
\begin{proposition}\label{Prop-higher}
Under the assumptions of Theorem \ref{thm:main}, equation (\ref{deriv-equal}) holds for $ |\beta|\geq 3 $.
\end{proposition}

We take the symmetric tensor fields $ T_{m}^{\alpha_1,\cdots,\alpha_J} $ and the operator $ T_{m} ^J $ to  be
\begin{equation*}
    \begin{aligned}
        T_{m}^{\alpha_1,\cdots,\alpha_J} (x)&:=Q_{\alpha_1}\cdots Q_{\alpha_J}  a_m(t,x,0)\\
        T_{m}^J(u_1,\cdots,u_J)(x)&:=T_{m}^{\alpha_1,\cdots,\alpha_J}(x)P_{\alpha_1}u_1(x) P_{\alpha_2}u_2(x) \cdots P_{\alpha_J}u_J(x).
    \end{aligned}
\end{equation*}
We consider $ T_{0}^{J}:=T_{1}^{J}  -T_{2}^{J} $. Then Proposition \ref{Prop-higher} reduces to $ T_{0}^{L}\equiv 0 $ for all $ L\geq 3 $.

We consider the boundary values in the form of $ f_\epsilon=t+\sum_{k=1}^{L+1}\epsilon^k h_k $ and repeat the linearization as in (\ref{eq:higher-order-0}) and (\ref{eq:higher-order-0-DN}) for $ m=1,2 $. The linearized equation reads
\begin{equation}\label{eq:higher-order}
    \left\{\begin{aligned}
        \nabla^* (a(t,x,0) \nabla w_{m} ^{h_1,\cdots,h_{L+1}})&=-\sum\limits_{\sigma\in\pi(L+1)}\nabla^*\big(\phi_{m}^{L}(h_{\sigma(1)},\cdots,h_{\sigma(L+1)}) &\\
        &\qquad +T_{m} ^L( w^{h_{\sigma(1)}},\cdots,  w^{h_{\sigma(L)}})\nabla w^{h_{\sigma(L+1)}}\big)&\quad\text{\rm in}~M,\\
        w_{m} ^{h_1,\cdots,h_{L+1}} & =0 & \quad\text{\rm on}~\partial M,
    \end{aligned}\right.
\end{equation}
where $ a(t,x,0):=a_1(t,x,0)=a_2(t,x,0)$ by Proposition \ref{Prop-0th}, while the corresponding DN map is
\begin{equation}\label{eq:higher-order-0.5}
    \begin{split}
        D^{L+1} N_{\gamma_m} (h_1,\cdots,h_{L+1})=\nu\cdot \bigg[&
            a(t,x,0)\nabla w_{m} ^{h_1,h_2,\cdots,h_{L+1}}\\
            &+\sum_{\sigma\in\pi(L+1)}\big( \phi_{m}^{L}(h_{\sigma(1)},\cdots,h_{\sigma(L+1)})\\
            &+T_{m} ^L(w^{h_{\sigma(1)}},\cdots, w^{h_{\sigma(L)}})\nabla w^{h_{\sigma(L+1)}}\big)
            \bigg]\quad \text{\rm on}~\partial M,
    \end{split}
\end{equation}
where $ \phi_{m}^L $, by definition in (\ref{def:T-phi}), only depends on $ T_m^{J-1}  $ and $ w_{m}^{h_1,h_2,\cdots,h_J} $ for $ J\leq L $.

We prove Proposition \ref{Prop-higher} by induction on $ L\geq 3 $, while the case $ L=3 $ has been proved in Proposition \ref{Prop-2nd}. The inductive hypothesis is that 
$$ \left\{\begin{aligned}
    &T_0^{J-1}\equiv 0\\
    &w_{0}^{h_1,h_2,\cdots,h_J}:= w_{1}^{h_1,h_2,\cdots,h_J}-w_{2}^{h_1,h_2,\cdots,h_J}\equiv 0
\end{aligned}\right.
\quad\mbox{for $ J\leq L $}.
$$
This implies $ \phi_{1}^L=\phi_{2}^L$. Consequently, taking the difference of (\ref{eq:higher-order}) and (\ref{eq:higher-order-0.5}) for $ m=1,2 $ gives that
\begin{equation}\label{eq:higher-order-1}
    \begin{cases}
        \nabla^* \left(a(t,x,0) \nabla w_{0} ^{h_1,\cdots,h_{L+1}}\right)=-\sum\limits_{\sigma\in\pi(L+1)}\nabla^*\left(T_0 ^L(w^{h_{\sigma(1)}},\cdots, w^{h_{\sigma(L)}})\nabla w^{h_{\sigma(L+1)}}\right)&\text{\rm in}~M,\\
        w_{0} ^{h_1,\cdots,h_{L+1}}=0&\text{\rm on}~\partial M,
    \end{cases}
\end{equation}
and that
\begin{equation}\label{eq:higher-order-1.5}
    0=\nu\cdot \left(a(t,x,0)\nabla w_{0} ^{h_1,\cdots,h_{L+1}}+\sum\limits_{\sigma\in\pi(L+1)}T ^L(w^{h_{\sigma(1)}},\cdots,w^{h_{\sigma(L)}})\nabla w^{h_{\sigma(L+1)}}\right)~\text{\rm on}~\partial M.
\end{equation}
In analogy with (\ref{eq:2nd-order-IBP}), we can get that for any $ h_1,\cdots,h_{L+2} \in C^\infty(\partial M) $,
\begin{equation}\label{eq:higher-order-2}
    \sum_{\sigma\in\pi(L+1)}\int_M T ^L(w^{h_{\sigma(1)}},\cdots, w^{h_{\sigma(L)}}) \langle\nabla w^{h_{\sigma(L+1)}},\nabla w^{h_{L+2}}\rangle =0.
\end{equation}

It suffices to prove the following proposition:

\begin{proposition}\label{hypo:int.eq.}
    Let $ J\geq 2 $. If the functions $ S^{\alpha_1,\cdots,\alpha_J}(x) $ are symmetric in indices $ \{\alpha_K\}_{K=1}^{J} $, and the equation
    \begin{equation}\label{eq:induction}
        \sum_{\sigma\in\pi(J+1)} \int_{M} S^{\alpha_1,\cdots,\alpha_J}P_{\alpha_1}u_{\sigma(1)} P_{\alpha_2}u_{\sigma(2)} \cdots P_{\alpha_J}u_{\sigma(J)}\langle\nabla u_{\sigma(J+1)},\nabla u_{J+2} \rangle=0
    \end{equation}
    holds for any
    $$ u_l \in \{u\in C^{\infty} (\overline{M}): \nabla^*(a(t,x,0)\nabla u) = 0 \},\quad l=1,\cdots,J+2, $$
    then every $ S^{\alpha_1,\cdots,\alpha_J} (x)$ vanishes on $ M $.
\end{proposition}
\begin{proof}
We prove this proposition by induction on $ J\geq 2 $. The statement is true for $ J=2 $ by adapting the proof of Proposition \ref{Prop-2nd}. Now we want to prove that the statement holds for $ J = L $ if it is true for $ J = L-1,~L\geq 3 $.

First we show that $ S^{0,\alpha_2,\cdots,\alpha_L}\equiv 0 $. Taking $ u_1=1  $ in (\ref{eq:induction}) for $ J=L $ gives
\begin{equation*}
    L\sum_{\sigma\in\pi(L)} \int_M S^{0,\alpha_2,\cdots,\alpha_L} P_{\alpha_2}u_{\sigma(1)+1} \cdots P_{\alpha_L}u_{\sigma(L-1)+1}\langle \nabla u_{\sigma(L)+1},\nabla u_{L+2} \rangle=0,
\end{equation*}
as $ \nabla u_1=0 $. By induction hypothesis, we know that $ S^{0,\alpha_2,\cdots,\alpha_L}\equiv 0 $.

By symmetry of indices we know that $ S^{\alpha_1,\cdots,\alpha_L}\equiv 0 $ if $ \prod_{J=1}^{L}\alpha_J=0 $. Putting this back to (\ref{eq:induction}), we obtain
\begin{equation}\label{eq:induction-1}
        \sum_{\sigma\in\pi(L+1)} \int_{M} S^{l_1,\cdots,l_L}\partial_{l_1}u_{\sigma(1)} \partial_{l_2}u_{\sigma(2)} \cdots \partial_{l_L}u_{\sigma(L)}\langle\nabla u_{\sigma(L+1)},\nabla u_{L+2} \rangle=0.
\end{equation}

It remains to prove $ S^{l_1,\cdots,l_L}\equiv 0$ for $1\leq l_1,\cdots,l_L\leq n $. For simplification, we let $ S^L $ denote the $ (0,L) $-tensor
\begin{equation*}
    S^L(X^{(1)} ,\cdots,X^{(L)}):=S^{l_1,\cdots,l_L}X_{l_1}^{(1)}\cdots X_{l_L}^{(L)}.
\end{equation*} 
We can still choose the following four harmonic functions in $ (M,g) $:
\begin{equation*}
    u_{L-1}:=w^{1},~u_{L}:=w^{2},~u_{L+1}:=w^{3},~u_{L+2}:=w^{4},   
\end{equation*}
where $ w^{k},~k=1,2,3,4 $ is given in (\ref{eq:cgo-choice}). Inserting the four functions above into the equation (\ref{eq:induction-1}), we get that
\begin{equation}
    0=\tau^{4} \int_{M}\left|b_{(0)}^{(1)}\right|^2\left|b_{(0)}^{(2)}\right|^2 e^{2\text{\rm Im}\lambda x^1-2\text{\rm Re}\lambda\psi^{(1)} (x')} U(\nabla\psi^{(1)},\nabla\psi^{(2)},\nabla u_{{1} },\cdots,\nabla u_{{L-2} } )+O(\tau^{3} ),
\end{equation}
where $ U(\nabla\psi^{(1)},\nabla\psi^{(2)},\nabla u_{{1} },\cdots,\nabla u_{{L-2} } ) $ is given by
\begin{displaymath}
    \begin{aligned}
        \frac{U(\cdots)}{L!}:=&S^L(\partial_1+i\nabla\psi^{(1)},\partial_1-i\nabla\psi^{(1)},\nabla u_{1},\cdots,\nabla u_{{L-2}}) \langle\partial_1+i\nabla\psi^{(2)},\partial_1-i\nabla\psi^{(2)}\rangle\\
        &+S^L(\partial_1+i\nabla\psi^{(1)},\partial_1+i\nabla\psi^{(2)},\nabla u_{1},\cdots,\nabla u_{{L-2}}) \langle \partial_1-i\nabla\psi^{(1)},\partial_1-i\nabla\psi^{(2)} \rangle \\
        &+S^L(\partial_1-i\nabla\psi^{(1)},\partial_1+i\nabla\psi^{(2)},\nabla u_{1},\cdots,\nabla u_{{L-2}}) \langle \partial_1+i\nabla\psi^{(1)},\partial_1-i\nabla\psi^{(2)}\rangle\\
        &+\sum_{j=1}^{L-2}\big[S^L(\partial_1+i\nabla\psi^{(1)},\partial_1-i\nabla\psi^{(1)},\partial_1+i\nabla\psi^{(2)},\nabla u_{1},\cdots,\widehat{\nabla u_{{j}}},\cdots,\nabla u_{{L-2}})\\
        &\qquad\qquad\cdot \langle \nabla u_{j},\partial_1-i\nabla\psi^{(2)} \rangle\big].
    \end{aligned}
\end{displaymath}
By Lemma \ref{lmm:4-Gaussian-beams}, $ U=0 $ at any $ x=(x^1,y_0)\in M $.

Now it becomes a problem of linear algebra to prove $ S^L=0 $. We denote by $ \xi^{(m)} $ the tangent vector $  \nabla\psi^{(m)}(y_0)=\dot{\Gamma}^{(m)}|_{\Gamma^{(m)}=y_0 } $, and by $ \eta_k $ the tangent vector $ \nabla u_{k}  $. At any $x=(x^1,y_0)\in M$, $ U=0 $ amounts to
\begin{equation}
    \begin{aligned}\label{eq:higher-order-3}
        0&=S^L(\partial_1+i\xi^{(1)},\partial_1-i\xi^{(1)},{\eta_1},\cdots,{\eta_{L-2}}) \langle\partial_1+i\xi^{(2)},\partial_1-i\xi^{(2)}\rangle \\
        &+S^L(\partial_1+i\xi^{(1)},\partial_1+i\xi^{(2)},{\eta_1},\cdots,{\eta_{L-2}}) \langle\partial_1-i\xi^{(1)},\partial_1-i\xi^{(2)}\rangle \\
        &+S^L(\partial_1-i\xi^{(1)},\partial_1+i\xi^{(2)},{\eta_1},\cdots,{\eta_{L-2}}) \langle\partial_1+i\xi^{(1)},\partial_1-i\xi^{(2)}\rangle  \\
        &+\sum_{j=1}^{L-2}\big[S^L(\partial_1+i\xi^{(1)},\partial_1-i\xi^{(1)},\partial_1+i\xi^{(2)},{\eta_1},\cdots,\widehat{{\eta_{j}}},\cdots,{\eta_{L-2}}) \langle{\eta_j},\partial_1-i\xi^{(2)}\rangle \big].
    \end{aligned}
\end{equation}

Notice that $ T_{(x^1,y_0)}M=T_{x^1}\mathbb{R}+ T_{y_0}M_0 $ results in
$$ \left(T_{(x^1,y_0)}M\right)^L=(T_{x^1}\mathbb{R})^L+(T_{y_0}M_0)^L+\sum_{\mathrm{sym}}T_{x^1}\mathbb{R}\times T_{y_0}M_0\times \left(T_{(x^1,y_0)}M\right)^{L-2}. $$
So we divide the proof of $ S^L=0 $ into the corresponding three parts:
\begin{enumerate}[(i)]
    \item $ S^L(\partial_1,\xi,{\eta_1},\cdots,{\eta_{L-2}})=0$ for any $ \xi\in T_{y_0} M_0 ,\eta_k\in T_{x}M $.
    \item $S^L(\partial_1,\partial_1,\cdots,\partial_1)=0.$
    \item $ S^L(\eta_1,\eta_2,\cdots,\eta_L)=0 $ for any $ \eta_k\in T_{y_0} M_0 $.
\end{enumerate}

Firstly we prove Case (i). Since $ \mathrm{span}V_{y_0}=T_{y_0}M_0 $, it suffices to prove (i) for $  \xi\in V_{y_0}$. Substitute $ -\xi^{(1)}  $ for $ \xi^{(1)}  $ in (\ref{eq:higher-order-3}) and take the difference. We get that
\begin{equation*}
    S^L(\partial_1,\xi^{(2)}-\langle \xi^{(1)},\xi^{(2)} \rangle\xi^{(1)},{\eta_1},\cdots,{\eta_{L-2}})=0,\quad\forall \xi^{(m)} \in V_{y_0}, \eta_k \in T_{x} M.
\end{equation*}
By similar methods for (\ref{eq:3rd-order-4}), we obtain that
\begin{equation}\label{eq:higher-order-4}
    S^L(\partial_1,\xi^{(1)},{\eta_1},\cdots,{\eta_{L-2}})=0,\quad\forall \xi^{(1)} \in V_{y_0}, \eta_k \in T_{x} M.
\end{equation}

For Case (ii), taking $ \eta_1=\cdots=\eta_{L-2}=\partial_1 $ in (\ref{eq:higher-order-3}) and applying (\ref{eq:higher-order-4}) gives
\begin{equation}\label{eq:higher-order-5}
    (L+2)S^L(\partial_1,\cdots,\partial_1)=0.
\end{equation}

Finally, it remains to prove Case (iii). Since $ V_{y_0}=\mathbb{S}^{n-2}   $, we can find $ \{\xi_l^{(1)}\}_{l=1}^{n-1} \subset V_{y_0} $ such that $$ T_{y_0} M_0=\text{\rm span}\{\xi_l^{(1)}\}_{l=1}^{n-1}. $$
Let $ \xi_{l,l'}^{(1)}\in V_{y_0} $ be parallel to $ \xi_{l}^{(1)}+\xi_{l'}^{(1)} $, and choose $ \xi^{(2)}\in V_{y_0}-\{\xi_{l,l'}^{(1)}\}_{l,l'=1}^{n-1}  $. Taking $ \eta_{L-2}=\partial_1,\xi^{(1)} =\xi_{l,l'}^{(1)}  $ in (\ref{eq:higher-order-3}) gives
\begin{equation*}
    S^L(\xi_{l,l'}^{(1)},\xi_{l,l'}^{(1)},\xi^{(2)},\eta_1,\cdots,\eta_{L-3} )=0,\quad\forall \eta_k \in T_{x} M.
\end{equation*}
Since $ \{\xi_l^{(1)}\}_{l=1}^{n-1} $ span $ T_{y_0} M_0 $ and $ S^L $ is symmetric, we obtain
\begin{equation}\label{eq:higher-order-5.5}
    S^L(\zeta_{1},\zeta_{2},\xi^{(2)},\eta_1,\cdots,\eta_{L-3} )=0,\quad\forall \zeta_1,\zeta_2\in T_{y_0} M_0, ~\eta_k \in T_{x} M.
\end{equation}
Taking $ \xi^{(1)} =\xi_{l,l'}^{(1)}  $ in (\ref{eq:higher-order-3}) and applying (\ref{eq:higher-order-5.5}) leads to
\begin{equation*}
    2S^L(\xi_{l,l'}^{(1)},\xi_{l,l'}^{(1)},\eta_1,\cdots,\eta_{L-2} )=0,\quad\forall \eta_k \in T_{x} M.
\end{equation*}
Again by using the fact that $ T_{y_0} M_0=\text{\rm span}\{\xi_l^{(1)}\}_{l=1}^{n-1} $ we know
\begin{equation}\label{eq:higher-order-6}
    2S^L(\zeta_{1},\zeta_{2},\eta_1,\cdots,\eta_{L-2} )=0,\quad\forall \zeta_l\in T_{y_0} M_0, \eta_k \in T_{x} M,
\end{equation}
which completes the proof of $ S^L=0 $.

To conclude, the statement of Proposition \ref{hypo:int.eq.} holds for $ J=L,~L\geq 3 $ if it holds for $ J=L-1 $. By induction on $ J $, we finish the proof for all $ J\geq 2 $.
\end{proof}


\begin{proof}[Proof of Proposition \ref{Prop-higher}]
    Assume $ T_0^{J-1}\equiv 0 $ and $ w_{0}^{h_1,h_2,\cdots,h_J}\equiv 0 $ for $ J\leq L $. By (\ref{eq:higher-order-2}), the coefficients $ T_0^{\alpha_1,\cdots,\alpha_J} $ of the operator $ T_0^J $ satisfy the conditions in Proposition \ref{hypo:int.eq.}, and hence
\begin{equation}
    T_0^{L} \equiv 0.
\end{equation}
Hence, the right-hand side of the elliptic equation (\ref{eq:higher-order-1}) vanishes in $ M $, which means by uniqueness of solution that
\begin{equation}
    w_0^{h_1,\cdots,h_{L+1}}=0.
\end{equation}
Therefore, we have proved the result for $ L > 3 $ from the induction hypothesis. 
\end{proof}
By far, we have proved (\ref{deriv-equal}).\par

\begin{proof}[Proof of Theorem \ref{thm:main}]
    If $ a_m $ is analytic in $p$, the Taylor series at $ p=0 $ converges. As we have shown in Proposition \ref{Prop-0th}, Proposition \ref{Prop-1st}, Proposition \ref{Prop-2nd} and Proposition \ref{Prop-higher}, all derivatives of $ a_m $ at $ p=0 $ are equal, and the choice of $ t $ does not influence this equality. By analytic continuation, equality of all derivatives at $ p=0 $ and any $ (s,x) $ implies $ a_1=a_2 $ globally.
\end{proof}

\bigskip

\noindent {\bf Acknowledgements.} The authors were supported in part by NSFC grant 12571451.  

\bigskip	\noindent {\bf Data Availability Statement.} Data sharing not applicable to this article as no datasets were generated or analysed during the current study.

\bigskip	\noindent {\bf Conflict of Interest.} The authors have no conflicts of interest to declare that are relevant to the content of this article.

\printbibliography{}

\end{document}